\documentclass[12pt]{amsart}
\usepackage{pdfsync}
\usepackage{a4,fullpage}

\usepackage{comment}
\usepackage{enumerate}
\usepackage{graphicx}
\usepackage{latexsym}
\usepackage{amsfonts}
\usepackage[utf8]{inputenc}
\usepackage{amsthm}
\usepackage{amsmath}
\usepackage{amssymb}
\usepackage[
            unicode,
            breaklinks=true,
            colorlinks=true]{hyperref}
\usepackage[active]{srcltx}
\usepackage{amscd}
\usepackage[all]{xy} 
\usepackage{mathrsfs} 
\usepackage{stmaryrd} 
\usepackage{amsopn} 
\usepackage{eepic}
\usepackage{epic}
\usepackage{eucal}
\usepackage{nicefrac}

\usepackage{wasysym}

\usepackage{epsfig}
\usepackage{psfrag}

\newtheorem{theorem}{Theorem}[section]
\newtheorem{lettertheorem}{Theorem}
\newtheorem{prop}[theorem]{Proposition}

\newtheorem{lemma}[theorem]{Lemma}

\newtheorem{cor}[theorem]{Corollary}

\def\I{(0,1)}
\def\tI{{t\in\I}}
\def\tab{{t\in(a,b)}}

\def\se{\hookrightarrow}
\def\to{\rightarrow}
\def\R{\mathbb{R}}

\def\e{{\rm Emb}_\subset(M,N)}
\def\et{{\rm Emb}_\subset^\sim(M,N)}
\def\mp{\mathcal{M}}
\def\mpfull{\mathcal{M}(M,N)}
\def\m{\mathcal{M}^\sim}
\def\mfull{\mathcal{M}^\sim (M,N)}
\def\mprk{\mathcal{M}(M,\R^k)}
\def\mrk{\mathcal{M}^\sim (M,\R^k)}

\def\V{\mathcal{V}}
\def\vplain{\mu_\V}
\def\v#1{\vplain\left(#1\right)}
\def\dv{\mathrm{d}\mu_\V}
\def\dn{\mathrm{d}\nuS}
\def\nuS{\nu_{\{\mathcal{S}_n\}}}
\def\vnu#1{\nuS\left(#1\right)}

\def\vs#1{\vplain\big((#1)\cap \mathcal{S}\big)}

\def\p{p^\dist_{\phi\psi}}

\def\pt{p^{d}_{\phi_t\phi}}

\def\dist{d}
\def\drk{\dist_{\R^k}}

\def\F{\mathcal{F}(\mp)}
\def\Fab{\F}
\def\Frk{\mathcal{F}(\mrk)}

\def\Feab{\mathcal{F}(\e)}
\def\Ft{\mathcal{F}(\m)}

\def\f{\{(\phi_t, B_t)\}}
\def\fab{\f_{t\in(a,b)}}

\def\ftildeab{\{(\widetilde{\phi_t}, \widetilde{B_t})_{\tab}\}}

\def\fI{\f_\tI}

\def\ftab{\{[\phi_t, B_t]\}_\tab}

\def\li{\varliminf}
\def\ls{\varlimsup}
\def\lic{\varliminf_{t\to c}}
\def\lsc{\varlimsup_{t\to c}}
\def\liz{\varliminf_{t\to 0}}
\def\lsz{\varlimsup_{t\to 0}}
\def\lia{\varliminf_{t\to a}}
\def\lsa{\varlimsup_{t\to a}}

\def\Sn{\{\mathcal{S}_n\}}
\def\D{\mathcal{D}}
\def\Dfull{\D_{\Sn}^d}
\def\Dfullrk{\D_{\Sn}^{\drk}}
\def\DS{\D^\dist_\mathcal{S}}
\def\DSrk{\D^{\drk}_\mathcal{S}}
\def\DSrkp{\D^{\dist_{\R^{k+1}}}_\mathcal{S}}
\def\DM{\D_M^\dist}
\def\DSone{\D^{\dist_1}_\mathcal{S}}
\def\DStwo{\D^{\dist_2}_\mathcal{S}}
\def\Snfull{\{\mathcal{S}_n\}_{n=1}^\infty}
\def\DSn{\D_{\{\mathcal{S}_n\}}}

\def\rfull{r^d_{\{\mathcal{S}_n\}}}

\def\mathclap#1{\text{\hbox to 0pt{\hss$\mathsurround=0pt#1$\hss}}}
\def\mc#1{\mathclap{#1}}

\newcommand{\N}{{\mathbb N}}

\newcounter{todo}
\renewcommand{\geq}{\geqslant}
\renewcommand{\leq}{\leqslant}

\newcommand{\rightd}{\xrightarrow{\;\mathcal{D}\;}}
\newcommand{\rightds}{\xrightarrow{\DS}}
\newcommand{\rightae}{\xrightarrow{\text{a.e.}}}
\newcommand{\rightdsrk}{\xrightarrow{\DSrk}}
\newcommand{\rightdsrkp}{\xrightarrow{\DSrkp}}

\newenvironment{remark}{\refstepcounter{theorem}\par\medskip\noindent{\bf
Remark~\thetheorem.}}{\unskip\nobreak\hfill\hbox{ $\oslash$}\par\bigskip}

\newenvironment{question}{\refstepcounter{theorem}\par\medskip\noindent{\bf
Question~\thetheorem}}{\unskip\normalfont \unskip\nobreak\hfill\hbox{\bigskip}}

\newenvironment{example}{\refstepcounter{theorem}\par\medskip\noindent{\bf
Example~\thetheorem.}}{\unskip\nobreak\hfill\hbox{ $\oslash$}\par\bigskip}

\newenvironment{definition}{\refstepcounter{theorem}\par\medskip\noindent{\bf
Definition~\thetheorem.}}

\begin{document}

\title{Metrics and convergence in the moduli spaces of maps}
\author{Joseph Palmer}

\begin{abstract}
 We provide a general framework to study convergence properties of families of maps. For manifolds $M$ and $N$ where $M$ is equipped with a volume form $\V$ we consider families of maps in the collection $\{(\phi, B) : B \subset M, \phi:B \rightarrow N\text{ with both measurable}\}$ and we define a distance function $\D$ similar to the $L^1$ distance on such a collection.  The definition of $\D$ depends on several parameters, but we show that the properties and topology of the metric space do not depend on these choices.  In particular we show that the metric space is always complete.  After exploring the properties of $\D$ we shift our focus to exploring the convergence properties of families of such maps.
\end{abstract}

\maketitle

\section{Introduction}\label{sec_intro}
The study of collections of maps between smooth manifolds, particularly of embeddings or diffeomorphisms, has recently attracted a lot of interest~\cite{Ab1997,Bi1996,Pe2007,PeVN2012,PeVN2012sharp}. Having a distance function defined on a collection of such mappings gives the collections the structure of a metric space about which new questions may be posed, as it is for instance done in ~\cite{PePRS2013}.

In ~\cite{PeVN2012} it is shown that if $M$ and $N$ are symplectic manifolds with $B_t\subset M$ for each $t\in (a,b)\subset \R$ and $$\{ (\phi_t, B_t) \mid t\in (a,b) \text{ and } \phi_t: B_t \rightarrow N \}$$ is a smooth (see Definition \ref{def_smooth}) family of symplectic embeddings such that
\begin{enumerate}
\item each $B_t$ is open and simply connected;
\item if $s<t$ than $\overline{B_t} \subset B_s$;
\item for all $t,s \in (a,b)$ the set $\bigcup\limits_{v \in [t,s]} \! \phi_v (B_v)$ is relatively compact in $N$,
\end{enumerate}
then there exists a symplectic embedding $$\phi_0 : \bigcup_{t\in (a,b)}\! B_t \hookrightarrow N.$$

This result starts with a collection of embeddings which do not necessarily converge and then assures the existence of an embedding from the union of their domains, which takes the place of the limit of these embeddings.  A natural next question is given some collection of embeddings which does not converge how much does each embedding need to be changed in order to get a collection which does converge.  In particular we are interested in situations in which each element of the collection must only be perturbed by an arbitrarily small amount in order to produce a new converging family, which is of course stronger than just requiring that an embedding of the union of their domains exist as in the result above.  In our case, again unlike in the result above, we are more interested in the nature of the family of embeddings than the existence of such a limiting embedding. This leads us to the problem of formalizing what we mean by a small perturbation. To address this we define a distance function on maps which do not necessarily have the same domain. Putting a metric on maps is exactly what is done when studying $L^p$ spaces, and once our distance is defined we will explain the relationship between our distance and the $L^1$ norm in Remark \ref{rmk_lp}.  Considering families of maps with different domains is absolutely essential for applications, see for instance the work of Pelayo-V\~{u} Ng\d{o}c ~\cite{PeVN2012,PeVN2012sharp}. Suppose that the maps are defined on subsets of a smooth manifold $M$ with a volume form $\V$ to a complete Riemannian manifold\footnote{In fact, we will soon see that the properties of the distance will not depend on the choice of metric and it is known that any smooth manifold admits a complete Riemannian metric, so we are not making any assumptions on $N$.} $N$ with natural distance $d$.  By this we mean that if $g$ is the Riemannian metric on $N$ and $y_1, y_2 \in N$ then $d(y_1,y_2) = \text{inf}\{\varint_0^1 \sqrt{g(\gamma'(t), \gamma'(t))}\,\mathrm{d}t\mid \gamma:[0,1]\to N \text{ is continuous with } \gamma(0)=y_1 \text{ and }\gamma(1)=y_2\}$. Throughout the paper by {\it metric} we will always mean a metric function on the space and if referring to a metric tensor we will always specify the Riemannian metric. Also, it is well known (see the Hopf-Rinow Theorem ~\cite[Satz I]{HR1931}) that $(N,g)$ being a geodesically complete Riemannian manifold is equivalent to $(N,d)$ being complete as a metric space so throughout this paper we will call such a manifold complete without specifying. Let $\vplain$ be the measure on $M$ induced by $\V$.  That is, for any $A \subset M$ we have $\v{A} = \varint_A \V$. Now we will define the set of maps we will be working with (shown in Figure \ref{fig_maps}).
\begin{definition}
Let $$\mpfull := \left\{(\phi,B_\phi):\begin{array}{l} B_\phi\subset M \text{ a nonempty measurable set and }\\ \phi:B_\phi \rightarrow N\text{ a measurable function}\end{array}\right\}$$ which we will frequently denote by $\mp$ when $M$ and $N$ are understood and we will also frequently write only $\phi$ where the domain is understood to be denoted by $B_\phi$. Also let $$\Fab=\big\{\f_{t\in(a,b)} \subset \mp\mid a,b\in\R \text{ with }a<b\big\}.$$
\end{definition}

\begin{figure}
 \centering
 \includegraphics[height=130pt]{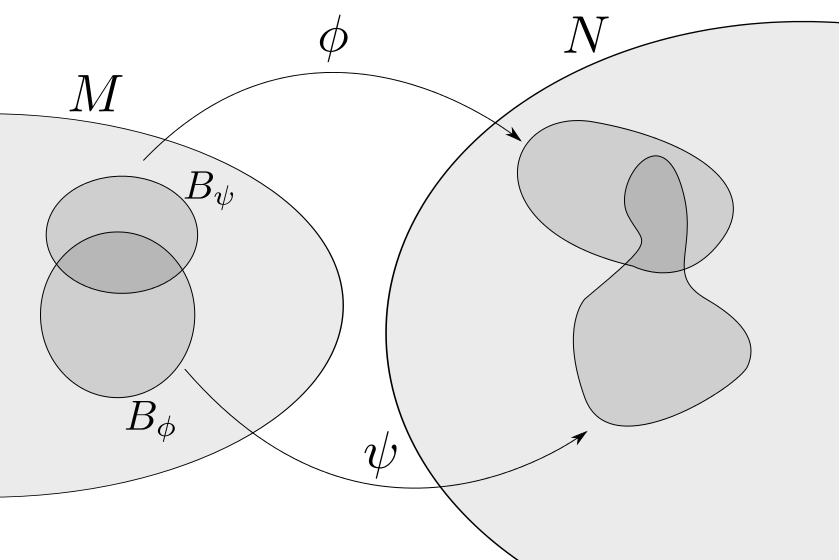}
 \caption{We will be considering maps from subsets of $M$ to $N$.}
 \label{fig_maps}
\end{figure}

In fact, for the remaining paper we will denote by $\mathcal{F}(\mathcal{S})$ the collection of one parameter families in a set $\mathcal{S}$ indexed by an open interval\footnote{Clearly it is equivalent to use any open interval, and thus we will use an arbitrary interval in the statements of the theorems but in the proofs we will often use $\I$ for simplicity.} in $\R$. 

A reasonable first guess for the ``distance'' between two elements in $\mp$ would be to integrate a penalty function over $M$.  That is we start with a function which assigns a penalty at each point in $M$ depending on how different the mappings are at that point, and then compute the ``distance'' between the two mappings by adding up all of these penalties via integration.  For each point in the symmetric difference, we know that one mapping acts on it while the other does not, so we assign it a maximum penalty of $1$.  For each point which is in the intersection of the domains, we simply find the distance between where each map sends the point, cut off to not exceed a maximum value of $1$, and use this as the penalty.  From this motivation we have with the following definition.
\begin{definition}\label{def_penalty}
For $(\phi,B_\phi), (\psi, B_\psi) \in \mp$ we define the {\it penalty function} $\p : M \to [0, 1]$ by $$\p (x) = \left\{ \begin{array}{ll} 1 & \text{if }x\in B_\phi \vartriangle B_\psi; \\ \text{min}\{ 1, d(\phi(x), \psi(x))\} & \text{if }x \in B_\phi \cap B_\psi; \\ 0 & \text{otherwise}\end{array}\right.,$$ and we define $$\DM((\phi, B_\phi), (\psi, B_\psi)) = \varint_M \p\dv.$$
\end{definition}
Notice that we need the minimum in the definition of $\p$ to make sure that any point on which both mappings act is not penalized more than the points which are only acted on by one mapping.  It is worth noting that even though the choice of the constant 1 may seem arbitrary it is shown in Proposition \ref{prop_da} that any positive constant may be used instead and the induced distance will still be strongly equivalent (see Definition \ref{def_metriceqv}).  Also, as long as $d$ is chosen so that the metric space $(N,d)$ is complete (which can always be done ~\cite[Theorem 1]{Nomizu}) the choice of $d$ will not change the properties of the induced metric.

However while $\DM$ is the natural ``distance'' it turns out to not be a distance function on $\mp$.  There are two main problems.  First, it is possible that $\DM$ will evaluate to zero on two distinct elements of $\mp$ and second it might be that $\DM$ evaluates to infinity.  The first problem is addressed easily by having $\DM$ act on equivalence classes of maps but the second problem will require a more delicate solution.

In fact, the problem of $\DM$ evaluating to infinity is even worse than it seems.  Suppose that $\phi_t(x)=(x,t)$ takes $\R$ into $\R^2$ for all $\tI$. Using the notation from above in this case we have that $M = B_{\phi_t} = \R$ for all $\tI$ and $N = \R^2$ with $\dist_{\R^2}$ the usual distance.   Then $\phi_t$ has a pointwise limit of $\phi_0 (x) := (x,0)$, but despite this we have that $\D_M^{\dist_{\R^2}}(\phi_t, \phi_0)$ is infinite for all $\tI$.  This example shows that $\DM$ is not always able to capture when a family of maps is converging.  We are able to solve this problem by observing $\DM$ restricted to various subsets of $M$.

\begin{definition}\label{def_ds}
We define $\D$ restricted to some set $\mathcal{S} \subset M$ of finite volume by $$\DS((\phi,B_\phi), (\psi, B_\psi)) = \varint_{S} \p\, \dv.$$ 
\end{definition}

\begin{figure}
 \centering
 \includegraphics[height=180pt]{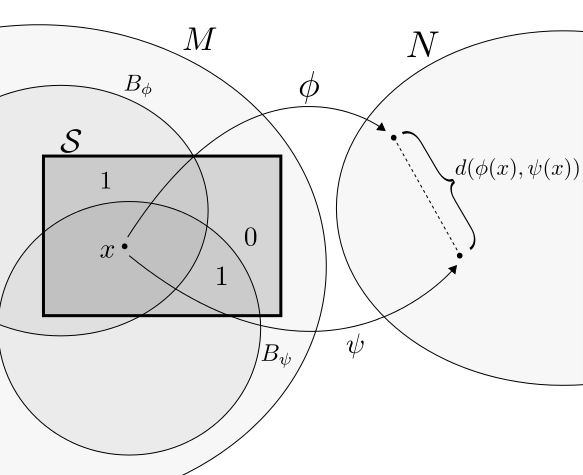}
 \caption{A graphic representing the values of $\p$ on $\mathcal{S} \subset M$.}
 \label{fig_dist}
\end{figure}

 Figure \ref{fig_dist} shows a good way to visualize computing $\DS$. Now each $\DS$ contains all of the information about $\DM$ on the set $\mathcal{S}$ and cannot evaluate to infinity.  The problem now, of course, is that we no longer have just a single metric with information about all of $M$ but instead have an infinite family of metrics which each have information about only one finite volume subset of $M$.  We solve this last problem by recalling that any manifold admits a nested exhaustion by compact sets, which must each have finite volume.  For the remaining portion of this paper by exhaustion we will always mean a countable nested exhaustion by finite volume sets.  In the following definition we set up the framework for this paper.  Corollary \ref{cor_dzero} states that part \eqref{part_dconv1} is well defined.
 
\begin{definition}\label{def_big}
Let $M$ and $N$ be manifolds with $\dist$ a metric on $N$ induced by a Riemannian metric.
\begin{enumerate}
 \item Let $\Snfull$ be a exhaustion of $M$ by nested finite volume sets\footnote{We know that such a collection must exist since each manifold admits a compact exhaustion} and let $\nuS$\footnote{We will write $\nuS$ in place of $\nu_{\Snfull}$ and $\Dfull$ in place of $\D_{\Snfull}^\dist$ for simplicity.} be the measure on $M$ given by $$\nuS(A) = \sum_{n=1}^{\infty} 2^{-n} \frac{\v{A \cap S_n}}{\v{S_n}}$$ for $A\subset M$.  Notice that $\nuS (M) = 1$ so $\nuS$ is a probability measure.  Then define\footnote{There is an equivalent definition of $\Dfull$ given in Proposition \ref{prop_ddef2} which is used in some of the proofs in this paper and explicitly shows the relation between $\Dfull$ and $\DS$.} $$\Dfull(\phi, \psi) = \varint_M \p\mathrm{d}\nuS.$$ 
 \item\label{part_dconv1} If $\Dfull (\phi, \psi) = 0$ for one choice of exhaustion then it equals zero for all choices of exhaustion and so in that case we write $\D(\phi, \psi)=0.$
 \item Let $$\mfull:=\mpfull / \sim$$ where $(\phi,B_\phi)\sim (\psi,B_\psi) \text{ if and only if } \D(\phi,\psi)=0.$ As before we will frequently shorten this to $\m$ and we denote by $[\phi, B_\phi]$ the equivalence class of $(\phi, B_\phi)\in\mp$.
\end{enumerate}
\end{definition}

Now we have enough notation to state our first result.
\pagebreak
\begin{lettertheorem}\label{thm_complete}
 Let $M$ and $N$ be manifolds and $\V$ a volume form on $M$.  Then for any choice of a metric $\dist$ on $N$ induced by a complete Riemannian metric and a countable exhaustion $\Snfull$ of $M$ by nested finite volume sets, the space $(\m,\Dfull)$ is a complete metric space.  Moreover, such a metric and exhaustion alway exist and if $d'$ and $\{\mathcal{S}_n'\}_{n=1}^\infty$ are other such choices then $\D_{\{\mathcal{S}_n'\}}^{d'}$ induces the same topology as $\Dfull$ on $\m$. 
\end{lettertheorem}

In light of Theorem \ref{thm_complete} we can now make the following definitions.

\begin{definition}
Let $a,b,c\in\R$ with $a<b$ and $c\in[a,b]$.  Also let $\fab \in \Fab$ and $\phi_0 \in \mp$.
\begin{enumerate} 
 \item Let $\mathcal{S} \subset M$ be any subset.  If $\lim_{t \rightarrow c} \DS (\phi_t, \phi_0) = 0$ we write $$\phi_t \rightds \phi_0\text{ as }t\to c.$$
 \item\label{part_dconv2} If $\lim_{t \rightarrow c} \Dfull (\phi_t, \phi)=0$ for one, and hence all, choices of $\{\mathcal{S}_n\}_{n=1}^\infty$ and $\dist$, we write $$\phi_t \rightd \phi_0\text{ as }t\to c.$$
 \item\label{part_ddist} Since all metrics $\Dfull$ generate the same topology on the set $\m$ we denote this set with such topology as $(\m, \D)$.
\end{enumerate}
\end{definition}

Thus $\m$ is a metric space with metric $\Dfull$ for any choice of exhaustion and complete metric and the metric spaces for different choices of exhaustion are all equivalent topologically\footnote{Here we should note that all of the information about $\Dfull$ is contained in $\DM$ if $M$ is finite volume, and in this case we will only have to consider $\DM$, see Remark \ref{rmk_mcomp}.}. 
\begin{remark}\label{rmk_lp}
Recall that $L^p$ spaces are collections of maps from a fixed measure set to $\R$.  Since  $\mp$ is a collection of all maps between fixed manifolds we can see that in some sense $\mp$ is a generalization of $L^p$ spaces.  The function $\Dfull$ is similar to the $L^1$ norm, but there are several differences.  It is noteworthy that {\it any} measurable mapping from $M$ to $N$ is ``integrable'', by which we mean that we can evaluate the distance between any two measurable mappings to get a finite number.  This is why we can let any such map be in $\mp$, as opposed to the case of $L^p$ spaces in which we must only consider integrable functions which have growth restrictions.   In Example \ref{ex_wave} we work out a specific case which does not converge in $L^p$ for any $p$ but does converge with respect to our distance.

There are many instances in which $L^p$ spaces have been generalized.  For example, many authors ~\cite{CUFiNe2003,DiRu2003,FaZh2001,KaLe2005} have explored generalizing $L^p$ spaces by letting $p$ be replaced by a function $p(x)$ which varies in the space.  These papers, though, still only consider the case of real valued functions.  In ~\cite{Am1990} the author studies functions with values in a metric space, as we do here, but he does not require any manifold structure and he only examines subsets of $\R^n$.  Finally, in ~\cite{St2010} the author studies $L^p$ functions on manifolds, but again these functions are required to take values in $\R$.  In all of these cases the authors are generalizing the important concept of $L^p$ functions, but only in our case can we examine all measurable functions between fixed manifolds and even functions with different domains.
\end{remark}

We are able to use the connection with $L^p$ spaces to prove a portion of Theorem \ref{thm_complete}. We use the well known result that  $L^1$ is complete to prove that as long as the target manifold $N$ is a complete Riemannian manifold we have that our metric space $\m$ is complete.  Since we have chosen $(N, d)$ to be complete the result follows.

Now that we have a metric defined on $\m$ we can explore families in $\Ft$ which converge with respect to that metric.  In Section \ref{sec_convg} we study another type of convergence and we explore the connection between these two natural forms of convergence on $\m$.
\begin{definition}\label{def_ptwsae}
 Let $a,b,c\in\R$ with $a<b$ and $c\in [a,b]$.  Let $\fab \in \F$ and suppose there exists some measurable $B\subset M$ satisfying $$B \subset \left\{x\in \lic B_t \mid \lim_{t \rightarrow c} \phi_t (x) \text{ exists}\right\}$$ and $\v{\lsc B_t \setminus B} = 0$\footnote{This in particular requires that the domains converge as sets as is described in Definition \ref{def_setconvg}.}. Then, with
 \begin{align*}
  \phi : B &\rightarrow N\\
         x &\mapsto \lim_{t \rightarrow c} \phi_t (x).
 \end{align*}
 we say that $\fab$ {\it converges to $(\phi, B)$ almost everywhere pointwise as $t\to c$} in $\mp$ and we write $\phi_t \rightae \phi$ as $t\to c$.
\end{definition}

We notice that if a family $\fab\in \F$ (for $a,b\in\R$ with $a<b$) converges to $\phi\in\mp$ almost everywhere pointwise as $t\to c\in[a,b]$ then it converges to $\phi$ in $\D$ as $t\to c$.  This gives us our second theorem.

\begin{lettertheorem}\label{thm_convg}
 Let $a,b,c \in \R$ such that $a<b$ and $c\in[a,b]$.  Suppose $\fab$ is a family such that $(\phi_t, B_t) \in \mp$ for each $t\in(a,b)$ and $(\phi,B) \in \mp$. If $\phi_t \rightae \phi$ as $t\to c$ then $\phi_t \rightd \phi \text{ as }t\to c$.
\end{lettertheorem}

Now that we have a good understanding of $(\m,\D)$ we will show one possible application of this metric. There are many different directions one could head from this point, but since there is research already being done regarding the convergence properties of families of embeddings ~\cite{PeVN2012,PeVN2012sharp} we will pursue an application in that field. We will use $\D$ to study families of embeddings which do not converge to an embedding and quantify how far they are from converging.  With this in mind we make the following definitions.

\begin{definition}
 Define $\e \subset \mp$ by $$\e = \{ (\phi, B) \in \mp \mid B\subset M \text{ is a submanifold and } \phi: B \se N \text{ is an embedding}\}$$ and define $\et \subset \m$ by $$\et = \{[\phi,B] \in \m \mid [\phi, B] \text{ has a representative in }\e\}.$$
\end{definition}
\begin{definition}\label{def_rpert}
 Let $a,b\in\R$ with $a<b$, $\varepsilon\geq0$, and $\fab\in\Fab$.  We say that a smooth family $\{(\widetilde{\phi_t}, \widetilde{B_t})\}_{t\in(a,b)}\in\Fab$ is a {\it convergent $\varepsilon$-perturbation (with respect to $\Dfull$)} of $\fab$ if
 \begin{enumerate}
  \item there exists $(\widetilde{\phi}, \widetilde{B})\in \e$ such that $\widetilde{\phi_t} \rightae \widetilde{\phi}$ as $t\to a$;
  \item\label{part_rdomain} $\li \widetilde{B_t} \subset \widetilde{B}$ and $B_t = \widetilde{B_t}$ for all $t\in(a,b)$;
  \item for all $t \in (a,b)$ we have that $\Dfull(\phi_t, \widetilde{\phi_t}) \leq \varepsilon$.
 \end{enumerate}
We define the {\it radius of convergence} of a family via
$$\rfull : \Fab \rightarrow [0,\infty]$$
 where $$\rfull(\fab ) := \text{inf}\left\{\varepsilon\geq 0 : \begin{array}{l} \text{there exists a smooth convergent}\\ \text{$\varepsilon$-perturbation of $\fab$}\end{array}\right\}.$$
\end{definition}
Notice in part \eqref{part_rdomain} of Definition \ref{def_rpert} we make some requirements on the domains.  This is so that we cannot simply remove from the domains a set of measure zero which includes the singular points. It is important to notice that, unlike many of the properties we have introduced so far, $\rfull$ does depend on the choice of $\dist$ and $\Snfull$. We are most interested in the $\rfull=0$ case, where an arbitrarily small perturbation can cause the family to converge to an embedding. It is natural to wonder whether a family can have radius of convergence zero but still not converge to any element of $\mp$ (including those elements which are not embeddings).  The following Theorem addresses this.

\begin{lettertheorem}\label{thm_r0convg}
 Let $a,b\in\R$ with $a<b$, $\f_{t\in(a,b)}$ be such that $(\phi_t, B_t) \in \mp$ for each $t\in(a,b)$, and let $\rfull$ be the radius of convergence function associated to a complete Riemannian distance $d$ on $N$ and an exhaustion of finite volume nested sets $\Snfull$ of $M$.  If $\rfull(\fab)=0$ then there exists $(\phi, B) \in \mp$ unique up to $\sim$ such that $\phi_t \rightd \phi$ as $t\to a$. Furthermore, the converse holds if there exists some $T\in(a,b)$ such that $s<t<T$ implies $B_s \subset B_t$.
\end{lettertheorem}

This theorem is important in the study of families with $\rfull=0$ because to characterize such families we may assume right away that there exists some limit $\phi_0$ and study its properties in order to understand the family we started with. In the final section we explore some ideas about the open questions about this function $\rfull$ including restricting to embeddings with specific properties and considering a converse of Theorem \ref{thm_r0convg} in the case in which the domains do not eventually shrink or stabilize.

\subsection{Outline of paper} In Section \ref{sec_def} we define the space of maps over which we will be working, we define the distance $\mathcal{D}$, and we prove several of its desirable properties including some parts of Theorem \ref{thm_complete}.  In Section \ref{subsec_prep} we prove a variety of Lemmas that will be needed in Section \ref{subsec_complete} to prove the rest of Theorem \ref{thm_complete}.  Next, in Section \ref{sec_convg} we examine the convergence properties of $\mathcal{D}$ and prove Theorem \ref{thm_convg}.  Finally, in Section \ref{sec_fam} we use what we have established in the preceding sections to study families of embeddings which do not converge to an embedding and prove Theorem \ref{thm_r0convg}. In our last section, Section \ref{sec_ques}, we comment on how the ideas from this paper can be used to further study such families.

\section{Definitions and preliminaries}\label{sec_def}
\subsection{Defining the distance}
Let $M$ be an orientable smooth manifold with volume form $\V$ and let $N$ be a smooth Riemannian manifold with natural distance function $\dist$. Again let $\vplain$ be the measure on $M$ induced by the volume form $\V$.  In this section we will prove all but the completeness statement in Theorem \ref{thm_complete}, which is postponed to Section \ref{sec_compl}. Recall the different notions of equivalent metrics.  The use of these terms varies, but for this paper we will use the following conventions.
\begin{definition}\label{def_metriceqv}
 Let $d_1$ and $d_2$ be metrics on a set $X$. Then we say that $d_1$ and $d_2$ are:
 \begin{enumerate}
  \item {\it topologically equivalent} if they induce the same topology on $X$;
  \item {\it weakly equivalent} if they induce the same topology on $X$ and exactly the same collection of Cauchy sequences;
  \item {\it strongly equivalent} if there exist $c_1, c_2 >0$ such that $$c_1 d_1 \leq d_2 \leq c_2 d_1.$$
 \end{enumerate}

\end{definition}

Now we define the following function.

\begin{definition}\label{def_da}
 Let $(\phi, B_\phi), (\psi, B_\psi) \in \mp$. For $\alpha > 0$ and a finite volume subset $\mathcal{S} \subset M$ define $$\mathcal{D}_\mathcal{S}^{d,\alpha} ((\phi, B_\phi), (\psi, B_\psi))= \varint_{\mathcal{S}} p^{d,\alpha}_{\phi\psi}\, \dv$$ where $$ p_{\phi\psi}^{d,\alpha}(x) = \left\{ \begin{array}{ll} \alpha & \text{if } x\in B_\phi \vartriangle B_\psi; \\ \text{min}\{ \alpha, d(\phi(x), \psi(x))\} & \text{if }x \in B_\phi \cap B_\psi; \\ 0 & \text{otherwise}.\end{array}\right.$$
\end{definition}

In Definition \ref{def_da} we have a family of functions depending on the choice of $\alpha>0$, but in fact these will induce strongly equivalent metrics.

\begin{prop}\label{prop_da}
 Let $\mathcal{S}$ be a finite volume subset of $M$.  If $\beta>\alpha>0$ then $$\mathcal{D}^{d,\alpha}_\mathcal{S} \leq \mathcal{D}^{d,\beta}_\mathcal{S} \leq \frac{\beta}{\alpha} \mathcal{D}^{d,\alpha}_\mathcal{S}.$$
\end{prop}

\begin{proof}
 Notice
 \begin{align*}
  \mathcal{D}_\mathcal{S}^{d,\alpha} (\phi, \psi) &= \varint_{\mc{B_\phi \cap B_\psi \cap \mathcal{S}}} \text{min}\{ \alpha, d(\phi, \psi)\}\, \dv + \alpha \mu_\V\big((B_\phi \vartriangle B_\psi) \cap \mathcal{S}\big)\\
  &\leq \varint_{\mc{B_\phi \cap B_\psi \cap \mathcal{S}}} \text{min}\{ \beta, d(\phi, \psi)\}\, \dv + \beta \mu_\V\big((B_\phi \vartriangle B_\psi) \cap \mathcal{S}\big)\\
  &= \mathcal{D}_\mathcal{S}^{d,\beta} (\phi, \psi)\\
 \end{align*}
 
 and also notice that
 \begin{align*}
  \mathcal{D}_\mathcal{S}^{d,\beta} (\phi, \psi) &= \varint_{\mc{B_\phi \cap B_\psi \cap \mathcal{S}}} \text{min}\{ \beta, d(\phi, \psi)\}\, \dv + \beta \mu_\V\big((B_\phi \vartriangle B_\psi) \cap \mathcal{S}\big)\\
  &\leq \varint_{\mc{B_\phi \cap B_\psi \cap \mathcal{S}}} \text{min}\{ \beta, \frac{\beta}{\alpha} d(\phi, \psi)\}\, \dv + \beta \mu_{\V}\big((B_\phi \vartriangle B_\psi) \cap \mathcal{S}\big)\\
  &= \frac{\beta}{\alpha} \mathcal{D}_\mathcal{S}^{d,\alpha} (\phi, \psi).
 \end{align*}
\end{proof}

So Proposition \ref{prop_da} means that the choice of $\alpha>0$ will not matter when we use $\D_{\mathcal{S}}^{d,\alpha}$ to define a metric, so henceforth we will assume that $\alpha = 1$.  That is, for any finite volume subset $\mathcal{S} \subset M$ we have $\DS$ as defined in Definition \ref{def_ds}. In the above proof we wrote out the definition of $\DS$ in a way which did not explicitly use the penalty function $\p$.  We can now notice that there is an equivalent definition of $\DS$ which will be useful for several of the proofs.  
\begin{prop}\label{prop_ddef2}
  Let $M$ and $N$ be manifolds with a volume form $\V$ on $M$, $\dist$ a distance on $N$ induced by a Riemannian metric, $\mathcal{S}\subset M$ a compact subset, and $\Snfull$ a nested exhaustion of $M$ by finite volume sets. The function $\DS$ given in Definition \ref{def_ds} can be written
 $$\DS (\phi, \psi) = \varint_{\mc{B_\phi \cap B_\psi \cap \mathcal{S}}} \text{min}\{1,\dist(\phi, \psi)\}\, \dv + \vs{B_\phi \vartriangle B_\psi}.$$
 and the function $\Dfull$ from Definition \ref{def_big} satisfies
 $$\Dfull (\phi, \psi) = \sum_{n=1}^{\infty} 2^{-n} \frac{\mathcal{D}^\dist_{\mathcal{S}_n}(\phi,\psi)}{\v{\mathcal{S}_n}}.$$
\end{prop}
This proposition has a trivial proof.  Before the next Proposition we have a definition.
\begin{definition}
 Suppose $a,b\in\R$ with $a<b$ and $c\in[a,b]$.  For a set $X$ and a function $$F: X \times X \to [0, \infty]$$ we say that a family $\{a_t\}_{t\in(a,b)}\subset X$ is {\it Cauchy with respect to $F$ as $t\to c$} if for all $\varepsilon>0$ there exists some $\delta>0$ such that $s,t\in(c-\delta, c+\delta)\cap (a,b)$ implies $F(a_t, a_s)<\varepsilon.$
\end{definition}
Below are several important properties of $\Dfull,$ which is defined in Definition \ref{def_big}.

\begin{prop}\label{prop_dconv}
 Let $a,b\in\R$ with $a<b$, $\fab \in \F$, and $\phi,  \psi \in \mp$. Further suppose that $d$ is a metric on $N$ induced by a Riemannian metric and $\Snfull$ is an exhaustion of $M$ by nested finite volume sets. The function $\Dfull$ has the following properties. 
 \begin{enumerate}
  \item $\fab$ is Cauchy with respect to $\Dfull$ as $t\to c$ iff it is Cauchy with respect to $\DS$ as $t\to c$ for all compact $\mathcal{S}\subset M$.
  \item $\lim_{t \rightarrow c} \Dfull(\phi_t, \phi) = 0$ if and only if $\phi_t \rightds \phi$ as $t\to c$ for all compact $\mathcal{S} \subset M$.
  \item $\Dfull(\phi, \psi) = 0$ if and only if $\DS (\phi, \psi) = 0$ for all compact $\mathcal{S} \subset M$ if and only if $\mu_\V\big((B_\phi \vartriangle B_\psi)\cap\mathcal{S}\big) = 0$ for every compact $\mathcal{S} \subset M$ and $\phi = \psi$ almost everywhere on $B_\phi \cap B_\psi$.
 \end{enumerate}
\end{prop}

\begin{proof}
 Let $\varepsilon>0$ and fix some compact subset $\mathcal{S} \subset M$.  Then $\mathcal{S} \subset \bigcup_{n=1}^\infty \mathcal{S}_n=M$ and since $\mathcal{S}$ has finite volume and the $\mathcal{S}_n$ are nested we can find some $I \in \N$ such that $\v{\mathcal{S} \setminus \mathcal{S}_I}< \varepsilon.$ This means that $$\DS \leq \D_{\mathcal{S}_I}^d+\varepsilon.$$ Now that we have this fact we will prove the three properties.
 
 (1) It is sufficient to assume that $a=c=0$ and $b=1$.  Suppose that $\fI$ is Cauchy with respect to $\Dfull$ as $t\to 0$ and fix some compact $\mathcal{S} \subset M$.  Let $\varepsilon>0.$
 
 From the above fact we can find some $I \in \N$ such that $\DS \leq \D_{\mathcal{S}_I}^d+\nicefrac{\varepsilon}{2}.$ Now, since this family is Cauchy with respect to $\Dfull$ we can find some $\delta \in \I$ such that $s,t < \delta$ implies $$\Dfull (\phi_t, \phi_s) < \frac{\varepsilon}{2^{I+1} \v{\mathcal{S}_I}}.$$  Using the expression for $\Dfull$ from Proposition \ref{prop_ddef2} we have that $$\sum_{n=1}^{\infty} 2^{-n} \frac{\mathcal{D}_{\mathcal{S}_n}^d(\phi_t,\phi_s)}{\v{\mathcal{S}_n}}<\frac{\varepsilon}{2^{I+1} \v{S_I}}$$ which in particular means $$ 2^{-I} \frac{\mathcal{D}_{\mathcal{S}_I}^d(\phi_t,\phi_s)}{\v{\mathcal{S}_I}}<\frac{\varepsilon}{2^{I+1} \v{S_I}} $$ so $\mathcal{D}_{\mathcal{S}_I}^d (\phi_t, \psi_t) < \nicefrac{\varepsilon}{2}$.
 
 Finally, we have that for $s,t<\delta$
 \begin{align*}
  \DS (\phi_t, \phi_s) &\leq \D_{\mathcal{S}_I}^d (\phi_t, \phi_s) + \frac{\varepsilon}{2}\\
  &< \varepsilon.
 \end{align*}
 The converse is easy and the proof of (2) is similar to the proof of (1).
 
 (3)  Suppose $\Dfull (\phi, \psi) = 0$ and fix some compact $\mathcal{S} \subset M$.  Notice that this means that $\D_{\mathcal{S}_n}^d(\phi,\psi)=0$ for all $n$. For any $\varepsilon>0$ from the fact above we know we can choose some $I$ such that
 \begin{align*}
  \DS (\phi, \psi) &\leq \D_{\mathcal{S}_I}^d(\phi, \psi)+\varepsilon\\
  &= \varepsilon
 \end{align*}
 so we may conclude that $\DS(\phi,\psi)=0.$
 
 Next, we assume that $\DS(\phi,\psi)=0$ for all compact $\mathcal{S} \subset M$.  Clearly this implies that $\vs{B_\phi \vartriangle B_\psi} =0$ because this is a term in $\DS$.  Suppose that there is some set of positive measure in $B_\phi \cap B_\psi$ for which $\phi \neq \psi$. Then since manifolds are inner regular there exists some compact subset of positive measure $K$ on which they are not equal.  But this implies that $\D_K^d (\phi, \psi)\neq 0$.
 
 Now since $\vs{B_\phi \vartriangle B_\psi} = 0$ for every compact $\mathcal{S} \subset M$ and $\phi = \psi$ almost everywhere on $B_\phi \cap B_\psi$ it is clear that $\Dfull (\phi, \psi) = 0$.
\end{proof}

\begin{cor}\label{cor_dzero}
 Let $\Snfull$ be an exhaustion of $M$ and let $\dist$ be a metric on $N$ induced by a Riemannian metric.  Suppose that $(\phi, B_\phi), (\psi, B_\psi)\in\mp$ such that $\Dfull (\phi, \psi) = 0$.  Then for any such parameters $\{\mathcal{S}_n'\}_{n=1}^\infty$ and $\dist'$ we have that $\D_{\{\mathcal{S}_n'\}}^{\dist'}(\phi, \psi)=0$ as well.
\end{cor}

Given the new information in Proposition \ref{prop_dconv} we can prove the following important Proposition.

\begin{prop}\label{prop_ddist}
For any choice of an exhaustion of $M$ by finite volume sets $\Snfull$ we have that $\Dfull$ is well defined and is a distance function on $\m$.  Also, if $\{\mathcal{S}_n'\}_{n=1}^\infty$ is another such choice of exhaustion then $\Dfull$ and $\mathcal{D}^\dist_{\{\mathcal{S}_n'\}}$ are weakly equivalent metrics on $\m$.
\end{prop}

\begin{proof}
Fix some $\Snfull$ a compact exhaustion of $M$ and let $\phi, \rho, \psi \in \mp$.  It is a straightforward exercise to show that $$\p (x) \leq p^d_{\phi\rho}(x) + p^d_{\rho\psi}(x)$$ for each $x\in M$ and thus $$\Dfull (\phi, \psi) \leq \Dfull (\phi, \rho) + \Dfull (\rho, \psi).$$ It should be noted that this inequality would not hold without the minimum in $\p$.  From here we can see that if $\phi \sim \rho$ then $$\Dfull (\phi, \psi) \leq \Dfull (\rho, \psi)$$ and similarly the opposite inequality is true as well.  So $$\Dfull (\phi, \psi) = \Dfull (\rho, \psi)$$ and thus $\Dfull$ is well defined on $\m$.

Now $\Dfull$ is positive definite on $\m$ because it is positive on $\mp$ and by definition $\Dfull (\phi, \rho)=0$ implies $\phi \sim \rho$.  Since $\Dfull$ is well defined on $\m$ and satisfies the triangle inequality on $\mp$ we know that it satisfies the triangle inequality on $\m$ and similarly we know that $\Dfull$ is symmetric on $\m$.

Proposition \ref{prop_dconv} parts (1) and (2) characterize both convergent and Cauchy sequences of $\Dfull$ in a way which is independent of the choice of $\Snfull$.  This means that different choices of $\Snfull$ will produce weakly equivalent metrics $\Dfull$.
\end{proof}

\subsection{Independence of Riemannian structure} We have seen that $\m$ is a metric space with metric $\Dfull$ for any choice of compact exhaustion and the metric spaces for different choices of exhaustion are all weakly equivalent.  Now we will show that this construction is actually independent of the choice of Riemannian metric on $N$ as well.  For the remaining portion of the paper we will use $\left\| \cdot \right\|$ to denote the usual norm in $\R^k$ and $\drk$ to denote the usual distance on $\R^k$.  

\begin{lemma}\label{lem_intr}
 Fix any finite volume subset $\mathcal{S}\subset M$ and let $a,b,c\in\R$ with $a<b$ and $c\in[a,b]$.  Now let $\fab \in \F$ and $(\phi, B)\in \mp$. Suppose that $\phi_t \rightds \phi\in\mp$ as $t\to c$ and $\mathcal{R}: B \cap \mathcal{S} \rightarrow (0, \infty)$ is any function.  Then $$\lim_{t \rightarrow c} \v{\{ x\in B_t \cap B \cap \mathcal{S} \mid d(\phi(x), \phi_t(x)) > \mathcal{R}(x)\}}=0.$$
\end{lemma}

\begin{proof}
 It is sufficient to prove for $a=c=0$ and $b=1$.  First, for $\tI$ let $C_t = \{ x\in B_t \cap B \cap \mathcal{S} \mid d(\phi(x), \phi_t(x)) > \mathcal{R}(x)\}.$ Since $C_t \subset \mathcal{S}$ we notice that
 \begin{align*}
  \DS (\phi_t, \phi) &\geq \varint_{C_t} \text{min}\{ 1, d(\phi, \phi_t)\}\, \dv\\
  &\geq \varint_{C_t} \text{min}\{1, \mathcal{R}\}\, \dv.
 \end{align*}
Now for each $n\in\N$ let $D^n = \{x \in B \cap \mathcal{S} \mid \mathcal{R}(x) > 2^{-n}\}$ and notice that
\begin{align*}
 \varint_{C_t} \text{min}\{1, \mathcal{R}\}\, \dv &\geq \varint_{\mc{D^n \cap C_t}} \text{min}\{1, \mathcal{R}\}\, \dv\\
 &\geq 2^{-n} \cdot \v{D^n \cap C_t}.
\end{align*}
Now combining the above facts we have that $\DS(\phi_t, \phi) \geq 2^{-n} \cdot \v{D^n \cap C_t}$ for any choice of $n\in\N$ so
\begin{equation}\label{eqn_limdn}
 \lim_{t \rightarrow 0} \v{D^n \cap C_t} = 0
\end{equation}
for all $n \in \N$.
 
Finally fix $\varepsilon>0$.  Since $\mathcal{R}(x) > 0$ for all $x\in B \cap \mathcal{S}$ we know that the collection $\{D^n\}_{n=1}^\infty$ covers $B \cap \mathcal{S}.$ Since $B \cap \mathcal{S}$ has finite volume we know there exists some $N \in \N$ such that $\v{(B \cap \mathcal{S}) \setminus D^N} < \nicefrac{\varepsilon}{2}.$ This implies that for all $\tI$ we have that $\v{C_t \setminus D^N}< \nicefrac{\varepsilon}{2}.$ By Equation \eqref{eqn_limdn} we conclude that we can choose some $T$ such that $t<T$ implies that $\v{C_t \cap D^N}<\nicefrac{\varepsilon}{2}.$  Now for $t<T$ we have that $\v{C_t} = \v{C_t \setminus D^N} + \v{C_t \cap D^N} < \varepsilon.$
\end{proof}

Now we show that any choice of continuous metric on $N$ will produce a weakly equivalent metric on $\m$.

\begin{lemma}\label{lem_equivmetric}
 Suppose that $d_1$ and $d_2$ are topologically equivalent metrics on $N$ and let $\Snfull$ be any exhaustion of $M$ by finite volume sets. Then $\DSn^{d_1}$ and $\DSn^{d_2}$ are topologically equivalent metrics on $\m$.
 \end{lemma}

\begin{proof}
 Fix finite volume $\mathcal{S} \subset M$.  If we show $\DSone$ and $\DStwo$ are topologically equivalent then we have proved the lemma by Proposition \ref{prop_dconv}. It is sufficient to show that the same families indexed by $\I$ converge so suppose $\fI\in\F$ and $(\phi_0,B_0)\in\mp$ such that $\phi_t \xrightarrow{\DSone} \phi_0$ as $t\to 0$ and we will show that $\phi_t \xrightarrow{\DStwo} \phi_0$ as $t\to 0$. Fix $\varepsilon>0$ and without loss of generality assume that $\varepsilon < \v{\mathcal{S}}.$
 Let $$C_t^2 = \left\{ x\in B_t \cap B_0 \cap \mathcal{S} \mid d_2(\phi_0(x), \phi_t(x)) > \frac{\varepsilon}{3 \v{\mathcal{S}}} \right\}.$$  Let $b_{y_0}^i (r) = \{y \in N \mid d_i(y, y_0)<r\}$ for $i=1,2.$ Since $d_1$ and $d_2$ are weakly equivalent metrics for each $y\in N$ there exists some radius $r_y>0$ such that the ball with respect to $d_1$ of radius $r_y$ centered at $y$ is a subset of the ball with respect to $d_2$ of radius $\nicefrac{\varepsilon}{3\v{\mathcal{S}}}$ centered at $y$. Thus there exists some $\mathcal{R}:B_0 \cap \mathcal{S} \rightarrow (0,\infty)$ such that 
 \begin{equation}\label{eqn_bd1d2}
  b_{\phi_0(x)}^1\big(\mathcal{R}(x)\big) \subset b_{\phi_0(x)}^2 \left(\frac{\varepsilon}{3\v{\mathcal{S}}}\right) \text{ for all }x\in B_0\cap\mathcal{S}.
 \end{equation}
 Define $C_t^1 = \{ x \in B_t\cap B_0 \cap \mathcal{S} \mid d_1(\phi_0 (x), \phi_t (x))>\mathcal{R}(x)\}$ and notice that Equation \eqref{eqn_bd1d2} implies that $C_t^2 \subset C_t^1.$  By Lemma \ref{lem_intr} since $\phi_t \xrightarrow{\DSone} \phi_0$ as $t\to 0$ we know that $\lim_{t \rightarrow 0} \v{C_t^1} = 0$ and so we can conclude that $$\lim_{t \rightarrow 0} \v{C_t^2} = 0.$$
 
 Now we can find some $T\in\I$ such that if $t<T$ then $\v{C_t^2} < \nicefrac{\varepsilon}{3}$ and also $\vs{B_t \vartriangle B_0}< \nicefrac{\varepsilon}{3}$.  Then
 \begin{align*}
  \DStwo (\phi_t, \phi_0) &= \varint_{\mc{B_t \cap B_0 \cap \mathcal{S}}} \text{min}\{1, d_2(\phi_t, \phi_0 \}\, \dv + \vs{B_t \vartriangle B_0}\\
  &\leq \varint_{\mc{(B_t \cap B_0 \cap \mathcal{S})\setminus C_t^2}} \text{min}\{1, d_2(\phi_t, \phi_0\}\, \dv + \varint_{\mc{C_t^2}} \text{min}\{1, d_2(\phi_t, \phi_0\}\, \dv + \vs{B_t \vartriangle B_0}\\
  &\leq \varint_{\mathcal{S}} \frac{\varepsilon}{3\v{\mathcal{S}}}\, \dv + \v{C_t^2} + \vs{B_t \vartriangle B_0}\\
  &< \nicefrac{\varepsilon}{3} + \nicefrac{\varepsilon}{3} + \nicefrac{\varepsilon}{3} = \varepsilon.
 \end{align*}
\end{proof}
We conclude this section with the following lemma.
\begin{lemma}\label{lem_equivriemmetric}
 Let $\Snfull$ be a nested exhaustion of $M$ by finite volume sets and suppose that $d_1$ and $d_2$ are metrics on $N$ induced by smooth Riemannian metrics.  Then $\DSn^{d_1}$ and $\DSn^{d_2}$ are topologically equivalent metrics on $\m$.
\end{lemma}

\begin{proof}
 Both $d_1$ and $d_2$ are continuous with respect to the given topology on $N$.  This means that they are topologically equivalent metrics and so by Lemma \ref{lem_equivmetric} the result follows.
\end{proof}

\begin{remark}\label{rmk_mcomp}
 If $M$ is finite volume, such as in the case that $M$ is compact, then there is an obvious preferred choice to make when choosing the exhaustion, namely simply $\{M\}$ itself. In such a case we will always use $$\D^d_M (\phi, \psi)= \varint_M \p \dv = \varint_{\mc{B_\phi \cap B_\psi}} \text{min}\{1, d(\phi, \psi)\} \dv + \v{B_\phi \vartriangle B_\psi}.$$ There are also no choices now when defining convergent $\varepsilon$-perturbations or the radius of convergence except for the choice of metric on $N$.
\end{remark}

\subsection{A representative example} To conclude this section will will work out an important example which will be referenced throughout the paper.

\begin{figure}[b]
 \centering
 \includegraphics[height=160pt]{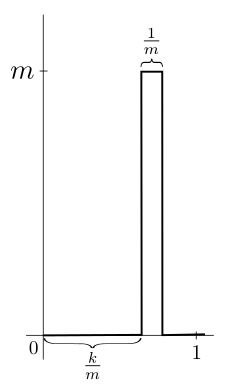}
  \caption{An image of $\Phi_{m,k}$.}
 \label{fig_genwave}
\end{figure} 

\begin{example}\label{ex_wave}
Let $\Phi_{m, k}: \I \to \R$ by $$\Phi_{m,k} (x) = m \cdot \chi_{(\nicefrac{k}{m},\nicefrac{k+1}{m})}(x)$$ (shown in Figure \ref{fig_genwave}) for $k, m \in \N$ with $k<m$ where $\chi_\mathcal{S}$ is the indicator function for the set $\mathcal{S}\subset \I$. We can see that $$\varint_{\mc{\I}} \Phi_{m,k} = 1$$ for all possible values of $k$ and $m$. We will use these functions to construct an example which is similar to the ``traveling wave'' example that is common in introductory analysis ~\cite{Folland} except that our example changes height so it always integrates to $1$.

Consider the sequence $$\phi_1 = \Phi_{0,1}, \phi_2 = \Phi_{0,2}, \phi_3 = \Phi_{1,2}, \phi_4 = \Phi_{0,3}, \phi_5 = \Phi_{1,3}, \phi_6 = \Phi_{2,3}, \phi_7 = \Phi_{0,4}, \ldots$$ (as shown in Figure \ref{fig_waves}) and let $$\phi_0 (x) = 0\text{ for all }x\in\I.$$\begin{figure}[h]
 \centering
 \includegraphics[height=130pt]{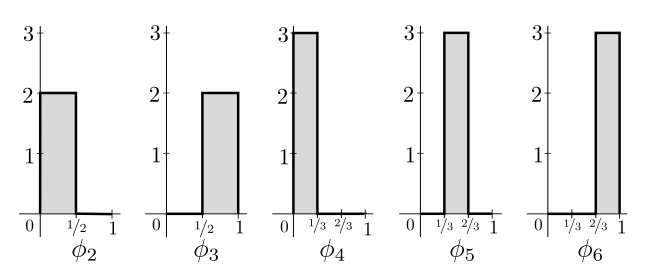}
 \caption{A few terms of $\{\phi_n\}$.  It can be seen that each integrates to 1 and the ``traveling waves'' pass over every point infinitely many times, so pointwise convergence is impossible.}
 \label{fig_waves}
\end{figure} Notice that this sequence does not converge pointwise to $\phi_0$ for any point $x\in\I$.  Also notice that since the integral of any element in this sequence is $1$ we can conclude that this sequence does not converge in $L^1$ (or $L^p$ for any $p\in[1,\infty]$) either (as is mentioned in Remark \ref{rmk_lp}), but it will converge with respect to $\D$.  This is because the measure of values in the domain which get sent to a number other than zero is becoming arbitrarily small, so we can conclude that $$\lim_{n \to \infty} \D_{\I}^{\dist_\R}(\phi_n, \phi_0) = 0.$$ This example shows a case in which we have a family which does not behave well pointwise almost everywhere or with respect to the $L^p$ norm, but it does behave well with respect to $\D$.

Of course, if we replace the indicator function with a bump function we can produce a sequence of smooth functions which has the same essential properties as these functions. In fact, for this example we have considered a sequence of functions instead of a continuous family of functions because it made it easier to describe the sequence, but we could easily extend this sequence to a smooth (see Definition \ref{def_smooth}) family of smooth embeddings of $\I$ into $\I \times \R$ indexed by $\tI$ which has the same properties.
\end{example}

\section{Completeness of \texorpdfstring{$\m$}{the metric space}}\label{sec_compl}
\subsection{Preparation}\label{subsec_prep}

Below is a collection of various technical Lemmas which are needed for Section \ref{subsec_complete}. In this section we will frequently use the alternative expression for $\D$ given in Proposition \ref{prop_ddef2}. 
\begin{lemma}\label{lem_l1comp}
 Let $A \subset M$ be a measurable finite volume set and let $a,b,c\in\R$ such that $a<b$ and $c\in[a,b]$. Suppose that some family of measurable functions $\{f_t: A \rightarrow \R^k\}_\tab$, is Cauchy with respect to $\varint_A \left\| f_t - f_s \right\|\, \dv$ as $t\to c$.  Then there exists some $f:A \rightarrow \R^k$ such that $f_t \xrightarrow{\mathcal{D}_\mathcal{A}^{\drk}} f$ as $t\to c$ where $\drk$ is the usual metric on $\R^k$.
\end{lemma}

\begin{proof}
 It is sufficient to show the result in the case that $a=c=0$ and $b=1$. For $\tI$ we know that $f_t$ maps into $\R^k$ so we may write it into components.  Write $$f_t(x) = (f^1_t(x), f^2_t(x), \ldots, f^k_t(x)).$$  Notice for any fixed $j\in\{1,2,\ldots,k\}$ that
 \begin{align*}
  \varint_A \left\| f_t - f_s \right\|\, \dv &= \varint_A \left( \sum_{i=1}^k (f_t^i - f_s^i)^2 \right)^{\nicefrac{1}{2}}\, \dv\\
  &\geq \varint_A \left| f_t^j - f_s^j \right|\, \dv
 \end{align*}
so we conclude that $\{f_t^j\}_\tI$ is Cauchy in $L^1(A)$ for each $j\in\{1,2,\ldots,k\}$.  Since $L^1(A)$ is complete we know for each $j\in\{1,2,\ldots,k\}$ there exists some function $f^j: A \rightarrow \R$ such that $$\lim_{t\rightarrow 0}\varint_A \left| f_t^j - f^j \right|\, \dv=0.$$ So we define $f (x) = (f^1(x), f^2(x), \ldots, f^k(x))$ for $x\in A$.  Now, notice that for any $x \in A$ we have $$\left\| f_t (x) - f(x) \right\| \leq \sum_{i=1}^k \left| f_t^i (x) - f^i (x) \right|.$$ Finally, notice
\begin{align*}
 \D_A^{\drk} (f_t, f) &= \varint_A \text{min}\{ 1, \left\| f_t - f\right\|\}\, \dv\\
 &\leq \varint_A \text{min}\{ 1, \sum_{i=1}^k \left| f_t^i - f^i \right|\} \, \dv\\
 &\leq \sum_{i=1}^k \varint_A \text{min}\{ 1, \left| f_t^i - f^i \right|\} \, \dv\\
 &\leq \sum_{i=1}^k \varint_A \left| f_t^i - f^i \right| \, \dv.
\end{align*}
Since $\varint_A \left| f_t^j - f^j \right| \, \dv$ goes to 0 as $t$ goes to 0 for any choice of $j\in\{1, 2, \ldots, k\}$ the result follows.
\end{proof}

\begin{lemma}\label{lem_closedlimit}
 Let $a,b,c\in\R$ with $a<b$ and $c\in[a,b]$.  Let $\fab \in \Frk$ and $(\phi, B) \in \mp$ be such that $\phi_t \rightd \phi$ as $t\to c$ and suppose there exists a fixed closed subset $P \subset \R^k$ such that $\phi_t (B_t) \subset P$ for all $\tab$. Then $$\v{\{x \in B \mid \phi (x) \notin P \}} = 0$$ and thus there exists some $(\phi',B') \sim (\phi,B)$ such that $\phi'(B) \subset P$.
\end{lemma}

\begin{proof}
 Without loss of generality assume that $a=c=0$ and $b=1$.  Since $P$ is closed notice that for $y\in \R^k$ we have that $$\inf_{p\in P} \{\drk (y, p)\}=0 \text{ implies } y \in P$$ where $\drk$ is the standard metric on $\R^k$. Thus, if we let $C = \{x\in B \mid \phi(x) \notin P \}$ and $$C_n = \{x \in B \mid \inf_{p\in P} \{\drk (\phi(x), p)\}>2^{-n}\}$$ for each $n\in\N$ then we have that $$C=\bigcup_{n=1}^\infty C_n.$$  So it will be sufficient to prove that $\v{C_n} = 0$ for each $n\in\N$.
 
 Let $\mathcal{S}\subset M$ be compact and notice that $\phi_t \rightd \phi$ as $t\to 0$ implies that $\lim_{t\rightarrow 0} \D_{C_n\cap \mathcal{S}}^{\drk} (\phi_t, \phi) = 0$.  We know
 \begin{align*}
 \D_{C_n\cap\mathcal{S}}^{\drk} (\phi_t, \phi) &= \varint_{\mc{B_t \cap C_n \cap \mathcal{S}}} \text{min}\{1, \drk (\phi_t, \phi) \}\, \dv + \v{(B_t \vartriangle B) \cap C_n\cap \mathcal{S}}\\
 &>2^{-n} \cdot \v{B_t \cap C_n \cap \mathcal{S}} + \v{(C_n \setminus B_t)\cap\mathcal{S}}\\
 &\geq 2^{-n} \cdot \v{C_n\cap\mathcal{S}}\geq 0.
 \end{align*}
This implies that $$\lim_{t\rightarrow 0} \left( 2^{-n} \cdot \v{C_n \cap\mathcal{S}}\right) = 0$$ for any choice of compact $\mathcal{S}\subset M$ which of course means $\v{C_n}=0$ for each $n\in\N$.
\end{proof}

\begin{lemma}\label{lem_dchecktop}
 Suppose that $\rho : N \rightarrow \R^k$ is an isometric embedding of Riemannian manifolds (ie, it preserves the metric tensor) and $a,b,c\in\R$ with $a<b$ and $c\in[a,b]$.  Then given some family $\fab \in \F$ and $\phi \in \mp$ we have that $\phi_t \rightd \phi$ as $t\to c$ if and only if $(\rho \circ \phi_t) \rightd (\rho \circ \phi)$ as $t\to c$.
\end{lemma}

\begin{proof}
 Let $d_N$ be the natural distance function on $N$ and let $\drk$ be the standard distance on $\R^k$. Then we may define a second distance function $d_2$ on $N$ by
 \begin{align*}
  d_2: N \times N &\rightarrow \R\\
  (y_1, y_2) &\mapsto \drk (\rho(y_1), \rho(y_2))
 \end{align*}
If we can show that these are topologically equivalent metrics on $N$ then the result will follow by Lemma \ref{lem_equivmetric}. Fix some $y_0 \in N$ and let $$b(r)=\{y \in N \mid d_N(y, y_0) < r\} \text{ and } b^2(r)=\{y \in N \mid d_2(y, y_0) < r\}.$$  Now notice that in general $d_2 \leq d_N$ (see Remark \ref{rmk_iso}), so we must only show that given some arbitrary $R>0$ we can find some $r>0$ such that $b^2(r) \subset b(R).$

Since $b(R) \subset N$ is an open set and $\rho$ is an embedding we can find some open set $U \subset \R^k$ such that $U \cap \rho(N) = \rho(b(R)).$ Now since $U$ is open and $\rho(y_0) \in U$ we can find some $r>0$ such that 
\begin{equation}\label{eqn_ball}
 \{z \in \R^k \mid \drk (z, \rho(y_0))<r \} \subset U.
\end{equation}

Now let $y \in b^2(r).$ Then we can see that Equation \eqref{eqn_ball} tells us that $\rho(y) \in U$. Clearly $\rho(y) \in \rho(N)$ so $\rho(y) \in U \cap \rho(N) = \rho(b(R)).$  Since $\rho$ is injective we now know that $y \in b(R).$  Thus $b^2(r) \subset b(R)$.

\end{proof}

\begin{remark}\label{rmk_iso}
 An isometric embedding of Riemannian manifolds preserves the metric at each point, so it will preserve the length of curves, but often the shortest path between two points in $\rho(N) \subset \R^k$ (a straight line) is not contained in $\rho(N)$.  This means that even though $\rho$ preserves the metric the images of two points in $\R^k$ may be closer than those two points are in $N$ and this is why $d_2 \leq d_N$ in the proof above.
\end{remark}

\subsection{Proof that \texorpdfstring{$(\m,\D)$}{the metric space} is complete}\label{subsec_complete}

The goal of this section is to prove that $(\m,\Dfull)$ is complete for any choice of an exhaustion of $M$ by finite volume sets $\Snfull$ and metric on $N$ induced by a complete Riemannian metric $\dist$.  To do this we first have to prove several lemmas.  We will start by considering mappings restricted to a compact set and indexed by $(0,1)$, but later it will be easy to generalize this to all of $M$ by using a compact exhaustion and to arbitrary intervals.  The first lemma proves the theorem in the special case that $N = \R^k$ and all maps have the same domain.

\begin{lemma}\label{lem_rkcomplete1}
 Fix some compact set $\mathcal{S} \subset M$ and let $\{(\phi_t, \mathcal{S})\}_\tI \in \mathcal{F}(\mprk)$ be a family which is Cauchy with respect to $\DSrk$ as $t\to 0$.  Then there exists some $\phi_0: \mathcal{S} \rightarrow \R^k$ unique up to $\sim$ such that $ \phi_t \rightdsrk \phi_0$ as $t\to 0$.
\end{lemma}

\begin{proof}
 The proof has five steps.  Figures \ref{fig_cauchyex} and \ref{fig_cauchyex2} show how the proof works in a specific case.
 
 {\it Step 1:} First we will define a new family $\{(\overline{\phi_t^n}, \mathcal{S})\}_\tI\in\F$ for each $n \in \N$. Since $\{(\phi_t,\mathcal{S})\}_\tI$ is Cauchy with respect to $\DSrk$ for each $n\in \N$ pick some $T_n \in \I$ such that 
 \begin{equation}\label{eqn_chooseTn}
  t \leq T_n \implies \DSrk(\phi_t, \phi_{T_n}) < 2^{-n}.  
 \end{equation}
Now for each $n \in \N$ we can define a new family $\{(\overline{\phi_t^n},\mathcal{S})\}_{t\in(0,T_n)}$ by $$\overline{\phi_t^n}(x) = \left\{\begin{array}{ll}
 \phi_t(x) &\text{if } \left\| \phi_t (x) - \phi_{T_n}(x) \right\| \leq \nicefrac{1}{2};\\ \\
 \phi_{T_n} (x) + \frac{\phi_t(x) - \phi_{T_n} (x)}{2 \left\|\phi_t(x) - \phi_{T_n} (x)\right\|} &\text{otherwise.}
 \end{array}
\right.$$

 \begin{figure}
 \centering
 \includegraphics[height=140pt]{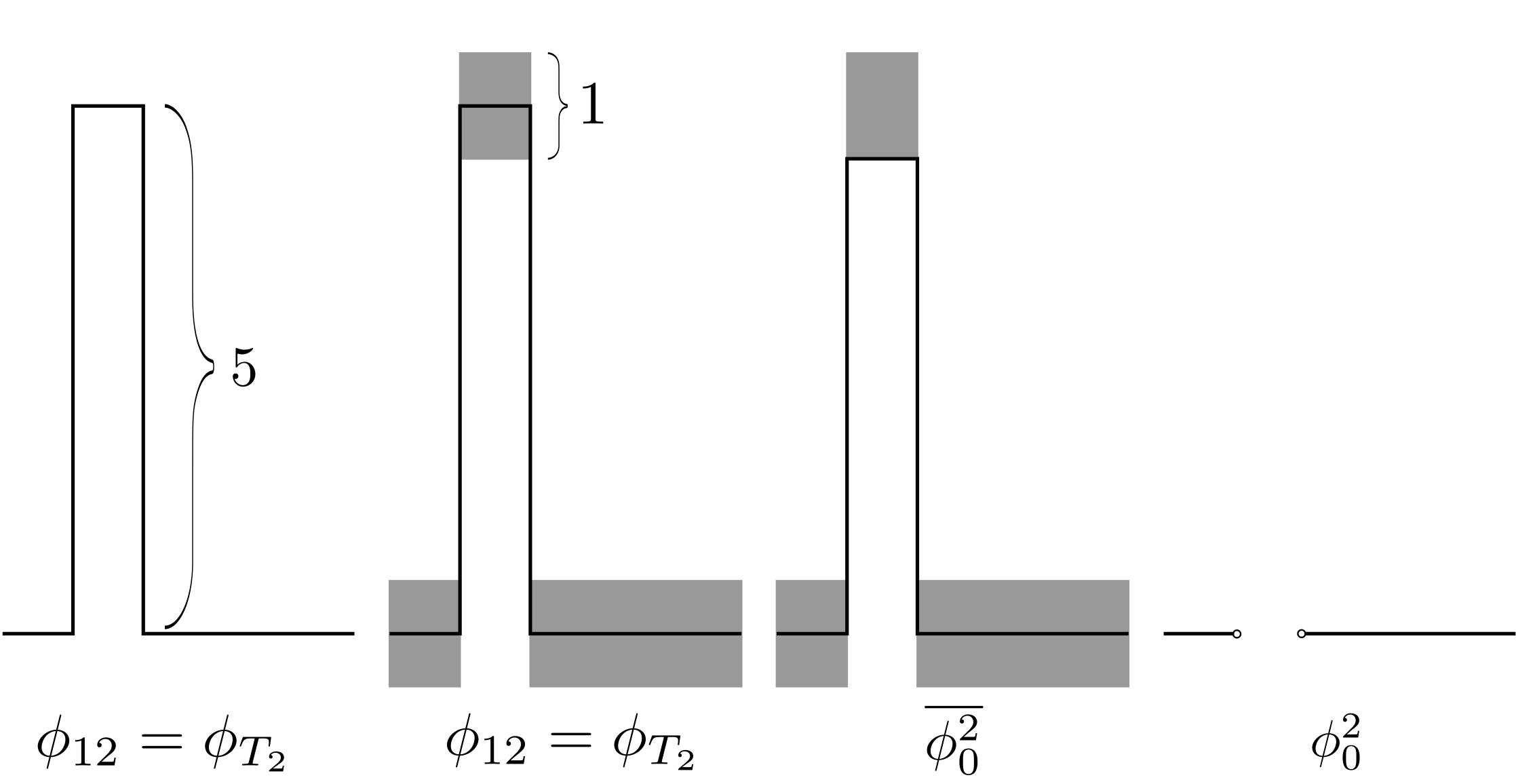}
 \caption{Applying the proof of Lemma \ref{lem_rkcomplete1} to Example \ref{ex_wave}. In Step 1 we choose $T_2 = 12$ because\protect\footnotemark in this case $\phi_{12}$ satisfies Equation \eqref{eqn_chooseTn} for $n=2$ and we restrict each mapping to have values within the shaded area (within a distance of $\nicefrac{1}{2}$ from $\phi_{T_2}$) to produce the family $\{(\overline{\phi_t^n}, \mathcal{S})\}.$ In Step 2 we find the limit of those functions to define $\overline{\phi_0^2}$. At the points in which this function takes values on the boundary of the shaded area we can see that the family is approaching a value outside of the shaded area, so in Step 3 we remove these points from the domain to form $\phi^2_0$.}
 \label{fig_cauchyex}
\end{figure}

 \begin{figure}
 \centering
 \includegraphics[height=150pt]{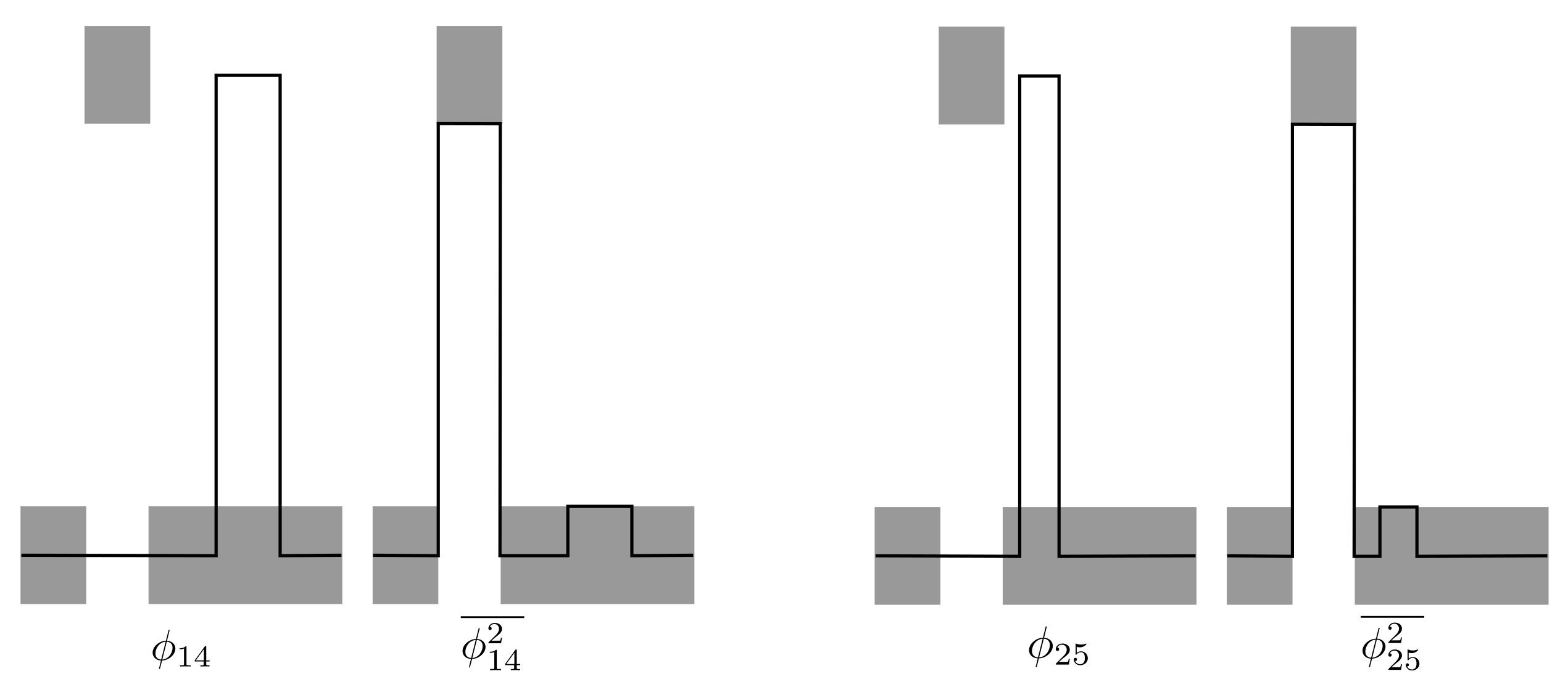}
 \caption{Two examples in which the maps are restricted to find the limit $\phi_0^2$ in Steps 1 and 2 of the proof of Lemma \ref{lem_rkcomplete1}.  In each case we start with $\phi_t$ and create $\overline{\phi_t}$ by changing the function to have only values with a distance less than $\nicefrac{1}{2}$ to $\phi_{T_2}$. }
 \label{fig_cauchyex2}
\end{figure}

 {\it Step 2:} Next we will show that each family $\{(\overline{\phi_t^n},\mathcal{S})\}_{t\in(0,T_n)}$ converges in $\DSrk$.  Notice for any $t,s < T_n$ we have that $\left\| \overline{\phi_t^n}(x) - \overline{\phi_s^n}(x) \right\| \leq 1$ so in fact we have that $$\DSrk (\overline{\phi_t^n}, \overline{\phi_s^n}) = \varint_\mathcal{S} \left\| \overline{\phi_t^n} - \overline{\phi_s^n} \right\|\, \dv.$$  Since for any $x \in \mathcal{S}$ we have that $\left\| \phi_t (x) - \phi_s(x) \right\| \geq \left\| \overline{\phi_t^n} (x) - \overline{\phi_x^n}(x)\right\|$ we know $$\DSrk (\phi_t, \phi_s) \geq \DSrk (\overline{\phi_t^n}, \overline{\phi_s^n})=\varint_\mathcal{S} \left\| \overline{\phi_t^n} - \overline{\phi_s^n} \right\|\, \dv$$ and since $\{(\phi_t, \mathcal{S})\}_\tI$ is Cauchy with respect to $\DSrk$ we now know that $\{(\overline{\phi_t^n},\mathcal{S})\}_{t\in(0,T_n)}$ is Cauchy with respect to $\varint_\mathcal{S} \left\| \overline{\phi_t^n} - \overline{\phi_s^n} \right\|\, \dv$. Thus by Lemma \ref{lem_l1comp} we know that for each $n \in \N$ there exists a map $\overline{\phi_0^n}: \mathcal{S} \rightarrow \R^k$ such that $\overline{\phi_t^n} \rightdsrk \overline{\phi_0^n}$ as $t\to 0$.
 
 \footnotetext{Recall that the functions in Example \ref{ex_wave} are labeled in the opposite order for convenience.}
 
 {\it Step 3:} In this step we will define $\phi_0^n$ for each $n$ on all but a subset of measure less than $2^{-n+2}$ of $\mathcal{S}$.  Let $$B_0^n = \{x\in \mathcal{S} \mid \left\| \phi_{T_n} (x) - \overline{\phi_0^n} (x) \right\|< \nicefrac{1}{4}\}.$$  Now define $$\phi_0^n = \overline{\phi_0^n} \vert_{B^n_0} : B^n_0 \rightarrow \R^k.$$ Now we will show that $\phi_0^n$ is defined on all but a small subset of $\mathcal{S}$.
 
 Let $\varepsilon > 0$ and pick some $t < T_n$ such that $\DSrk (\overline{\phi_t^n}, \overline{\phi_0^n}) < \varepsilon.$  Then
 \begin{align*}
  \DSrk (\phi_{T_n}, \overline{\phi_0^n}) &\leq \DSrk (\phi_{T_n}, \overline{\phi_t^n}) + \DSrk (\overline{\phi_t^n}, \overline{\phi_0^n})\\
  & < 2^{-n} + \varepsilon
 \end{align*}
for all $\varepsilon>0$ so we may conclude that $\DSrk (\phi_{T_n}, \overline{\phi_0^n}) \leq 2^{-n}$. Next also notice that since $\mathcal{S}\setminus B^n_0 \subset \mathcal{S}$ we know that
\begin{align*}
 \DSrk (\phi_{T_n}, \overline{\phi_0^n})& \geq \varint_{\mc{\mathcal{S} \setminus B^n_0}} \text{min}\{1,\left\| \phi_{T_n} - \overline{\phi_0^n}\right\|\}\, \dv\\
 &\geq \frac{1}{4} \v{\mathcal{S}\setminus B_0^n}.
\end{align*}
This means that $\v{\mathcal{S} \setminus B^n_0} \leq 2^{-n+2}.$ Since $B^n_0\subset \mathcal{S}$ we conclude that $\v{B_0^n} \geq \v{\mathcal{S}} - 2^{-n+2}.$ So if $$\v{\mathcal{S} \setminus \bigcup_{n=1}^\infty B_0^n}=\alpha>0$$ the we would have a contradiction because we can choose some $n\in\N$ such that $2^{-n+2}<\alpha$. Thus we have that $$\v{\mathcal{S} \setminus \bigcup_{n=1}^\infty B_0^n} = 0.$$

{\it Step 4:}  Next we must show that the limiting functions are equal on the overlap of their domains.  That is, we must show for any $m,n \in \N$ that $\phi_0^m(x) = \phi_0^n(x)$ for almost every $x \in B^m_0 \cap B^n_0$.
Our first step towards this goal is to define $$E_t^n = \{ x \in B^n_0 \mid \left\| \phi_{T_n} (x) - \overline{\phi_t^n}(x) \right\| \geq \nicefrac{1}{2} \}$$ and show that $$\lim_{t \rightarrow 0} \v{E^n_t} = 0$$ for all $n \in \N.$

Since $E_t^n \subset B^n_0$ we know that for any $x \in E_t^n$ we have that $$ \left\| \phi_{T_n} (x) - \overline{\phi_t^n}(x) \right\| \geq \nicefrac{1}{2}$$ and also that $$\left\| \phi_{T_n} (x) - \overline{\phi_0^n} (x) \right\| < \nicefrac{1}{4}.$$ Thus we may apply the triangle inequality to notice that $$\left\| \overline{\phi_t^n}(x) - \overline{\phi_0^n}(x) \right\| \geq \nicefrac{1}{4} \text{ for } x \in E_t^n.$$  Now we just notice that since $E_t^n \subset \mathcal{S}$ we have
\begin{align*}
 \DSrk(\overline{\phi_t^n}, \overline{\phi_0^n}) &\geq \varint_{E^n_t} \left\| \overline{\phi_t^n} - \overline{\phi_0^n} \right\|\, \dv\\
 &\geq \frac{1}{4} \v{E_t^n}.
\end{align*}
Thus we conclude that $\lim_{t \rightarrow 0} \v{E^n_t} = 0$, as desired.

Now let $C = \{x \in B^n \cap B^m \mid \phi_0^n (x) \neq \phi_0^m(x)\}$ and we will show that $\v{C} = 0$ to complete this step.  Notice that for any $x \in C \setminus (E_t^n \cup E_t^m)$ we have that $\overline{\phi_t^n}(x)=\overline{\phi_t^m}(x)=\phi_t(x)$.  Notice 
\begin{align*}
\varint_{C} \text{min}\{1,\left\| \overline{\phi_t^n}-\overline{\phi_t^m}\right\|\}\, \dv &\leq \v{E^n_t} + \v{E^m_t} + \varint_{\mc{C \setminus (E^n_t \cap E^m_t)}}\text{min}\{1,\left\| \overline{\phi_t^n}-\overline{\phi_t^m}\right\|\}\, \dv\\
&=\v{E^n_t} + \v{E^m_t}.
\end{align*}
Then, since $0<\varint_C \text{min}\{1,\left\| \overline{\phi_t^n}-\overline{\phi_t^m}\right\|\}\, \dv\leq \v{E^n_t} + \v{E^m_t}$ and the right side decreases to zero as $t\to 0$ we conclude that $$\lim_{t \rightarrow 0} \varint_C \text{min}\{1,\left\| \overline{\phi_t^n}-\overline{\phi_t^m}\right\|\}\, \dv=0$$ so $\lim_{t \rightarrow 0}\D_C^{\drk}(\overline{\phi_t^n},\overline{\phi_t^m})=0$.

Finally, by the triangle inequality $$\D_C^{\drk} (\phi_0^n, \phi_0^m) \leq \D_C^{\drk} (\phi_0^n, \overline{\phi_t^n})+\D_C^{\drk} (\overline{\phi_t^n},\overline{\phi_t^m})+\D_C^{\drk} (\overline{\phi_t^m},\phi_0^m)$$ and we know each term on the right goes to zero as $t\to 0$.  Since $t$ does not appear on the left side we may conclude that $$\D_C^{\drk} (\phi_0^n, \phi_0^m) =\varint_C \text{min}\{1,\left\| \phi_0^n  - \phi_0^m \right\|\} \, \dv = 0.$$ Notice that the function $\text{min}\{1,\left\| \phi_0^n (x) - \phi_0^m (x) \right\|\}$ is strictly positive on $C$, so since integrating it over $C$ yields zero we conclude that $\v{C} = 0$.

{\it Step 5:} In this step we will define the map $\phi_0 : \mathcal{S} \rightarrow \R^k$ and show that it is the unique limit.  Now define $$\phi_0 (x) = \phi_0^n (x) \text{ for any } n \text{ such that }x \in B_0^n.$$ This map is well defined almost everywhere because the $\phi_0^n$ are equal almost everywhere on the overlap of their domains and $\cup_{n=1}^\infty B^n_0$ covers almost all of $\mathcal{S}$.

Now we must show this is the limit.  Since we already know that $\{(\phi_t, \mathcal{S})\}_{\tI}$ is Cauchy it is sufficient to choose a subsequence and show it converges to $\phi_0$.  We will consider the sequence $\{(\phi_{T_n},\mathcal{S})\}_{n=1}^\infty$.  Fix some $\varepsilon >0$ and pick $N \in \N$ such that $2^{-n+2} < \nicefrac{\varepsilon}{3}$ for all $n>N$. Now for each $n>N$ pick $t_n\in(0,T_n)$ such that $\DSrk (\overline{\phi_{t_n}^n}, \overline{\phi_0^n}) < \nicefrac{\varepsilon}{3}$.  Then for any $n>N$ we have
\begin{align*}
 \DSrk (\phi_{T_n}, \phi_0) &\leq \DSrk (\phi_{T_n},\overline{\phi_{t_n}^n}) + \DSrk (\overline{\phi_{t_n}^n},\overline{\phi_0^n}) + \DSrk (\overline{\phi_0^n},\phi_0)\\
 &< 2^{-n} + \nicefrac{\varepsilon}{3} + 2^{-n+2}\\
 &< \varepsilon.
\end{align*}

Thus we conclude that $\phi_{T_n} \rightdsrk \phi_0$ as $t\to 0$ and thus $\phi_t \rightdsrk \phi_0$ as $t\to 0$.  To show that this is unique suppose that there exists some other $\phi_0':\mathcal{S} \rightarrow \R^k$ such that $\phi_t \rightdsrk \phi_0'$ as $t\to 0$.  Then for any compact set $\mathcal{S}'\subset M$ and $\tI$ we have that $$\D_{\mathcal{S}'}^{\drk}(\phi_0, \phi_0') \leq \DSrk(\phi_0, \phi_t) + \DSrk (\phi_t, \phi_0') \rightarrow 0\text{ as }t\to 0$$ since both have domain $\mathcal{S}$, so $\phi_0 \sim \phi_0'$.

\end{proof}
For the next step we will continue to focus on a single compact set and the case in which $N=\R^k$, but this time we will allow the domains of the functions to vary.
\begin{lemma}\label{lem_rkcomplete2}
 Fix some compact subset $\mathcal{S} \subset M$.  Any family $\fI\in\mathcal{F}(\mprk)$ which is Cauchy with respect to $\DSrk$ as $t\to 0$ also converges with respect to $\DSrk$ as $t\to 0$ to some $\phi_0: B_0 \rightarrow \R^k$ where $B_0 \subset \mathcal{S}.$ Moreover, among maps in $\mp$ with domains a subset of $\mathcal{S}$ that share this property, $\phi_0$ is unique up to $\sim$.
\end{lemma}

\begin{proof}
 Let $\fI \in \Frk$ be a family which is Cauchy as $t\to 0$ and define $\pi: \R^{k} \rightarrow \R^{k+1}$ via $\pi(x_1, \ldots, x_k) = (0, x_1, \ldots, x_k)$. Now, for each $\tI$ define $\widehat{\phi_t}:\mathcal{S} \rightarrow \R^{k+1}$ by $$\widehat{\phi_t}(x) = \left\{\begin{array}{ll} (1, 0, \ldots, 0) &\text{if } x\notin B_t\\ \pi(\phi_t(x))& \text{if } x \in B_t. \end{array}\right.$$
 
 Notice that for all $s,t\in\I$ $$\text{min}\{1, \left\| \widehat{\phi_t}(x) - \widehat{\phi_s}(x)\right\|\} = \left\{\begin{array}{ll} \text{min}\{1, \left\| \phi_t (x) - \phi_s(x) \right\| \} & \text{ if } x \in B_t \cap B_s\\ 1 & \text{ if } x \in B_t \vartriangle B_s \\ 0 & \text{ if } x \notin B_t \cup B_s. \end{array}\right.$$
Thus for $s,t\in\I$ we have that
\begin{align*}
 \DSrkp (\widehat{\phi_t}, \widehat{\phi_s}) &= \varint_\mathcal{S} \text{min}\{1, \left\| \widehat{\phi_t} - \widehat{\phi_s}\right\|\}\, \dv\\
 &=\varint_{\mc{B_t \cap B_s \cap \mathcal{S}}} \text{min}\{1, \left\| \widehat{\phi_t} - \widehat{\phi_s}\right\|\}\, \dv + \varint_{\mc{(B_t \vartriangle B_s) \cap \mathcal{S}}} 1 \, \dv + \varint_{\mc{\mathcal{S} \setminus (B_t \cup B_s)}} 0 \, \dv\\
 &=\varint_{\mc{B_t \cap B_s \cap \mathcal{S}}} \text{min}\{1, \left\| \widehat{\phi_t} - \widehat{\phi_s}\right\|\}\, \dv + \vs{B_t \vartriangle B_s}\\
 &=\DSrk(\phi_t, \phi_s).
\end{align*}
 Now we can see that $\{\widehat{\phi_t}, \mathcal{S}\}_\tI$ must be Cauchy as well.  Since these are all functions into $\R^{k+1}$ with the same domain we can invoke Lemma \ref{lem_rkcomplete1} to conclude that there exists some limit $\widehat{\phi_0}: \mathcal{S} \rightarrow \R^{k+1}$ which is unique up to $\sim$ such that $\widehat{\phi_t} \rightdsrkp \widehat{\phi_0}$ as $t\to 0$. Since $K = \pi(\R^k) \cup \{(1, 0, \ldots, 0)\}$ is a closed subset of $\R^{k+1}$ we can invoke Lemma \ref{lem_closedlimit} to conclude that we may assume that $\widehat{\phi_0}(\mathcal{S}) \subset K$.
 
 This allows us to define $(\phi_0,B_0)$ in the following way.  Let $$B_0 = \{x \in \mathcal{S} \mid \widehat{\phi_0}(x) \neq (1, 0, \ldots, 0)\}.$$ So for any $x \in B_0$ we know that $\widehat{\phi_0}(x) \in \pi (\R^k)$, which means that we can define
 \begin{align*}
  \phi_0: B_0 &\rightarrow \R^k\\
  x &\mapsto \pi^{-1} (\widehat{\phi_t}(x)).
 \end{align*}
 Since $\DSrk(\phi_t, \phi_0) = \DSrkp(\widehat{\phi_t}, \widehat{\phi_0})$ we can see that $\widehat{\phi_t} \rightdsrkp \widehat{\phi_0}$ implies that $\phi_t \rightdsrk \phi_0$ and we know that $\phi_0$ is unique up to $\sim$ because $\widehat{\phi_0}$ is.

\end{proof}
Finally we will expand to consider all of $M$ instead of just a single compact set, but we will still only consider $N=\R^k$.
\begin{lemma}\label{lem_rkcomplete3}
 $(\mrk, \D)$ is complete.  That is, for any choice of a nested finite exhaustion of $M$ denoted $\Snfull$ we have that $\big(\mrk, \D_{\Sn}^{\drk}\big)$ is a complete metric space where $\drk$ is the standard metric on $\R^k$.
\end{lemma}
\begin{proof}
 It is sufficient to show that families indexed by $\I$ which are Cauchy as $t\to 0$ also converge as $t\to 0$.  Let $\fI \in \Frk$ be Cauchy with respect to $\Dfullrk$. From Proposition \ref{prop_dconv} we know that this means for all compact $\mathcal{S} \subset M$ this sequence is Cauchy with respect to $\DSrk$ and from Lemma \ref{lem_rkcomplete2} we know that this means for each compact $\mathcal{S}\subset M$ we have some $\phi_0^\mathcal{S}: B_0^\mathcal{S} \to \R^k$ where $B_0^\mathcal{S}\subset\mathcal{S}$ such that $(\phi_0^\mathcal{S}, B_0^\mathcal{S})$ is unique up to $\sim$. Let $\Snfull$ be a nested compact exhaustion of $M$ and now we would like to conclude that for $n<m$ we have that $$(\phi_0^{\mathcal{S}_m} \vert_{\mathcal{S}_n}, B_0^{\mathcal{S}_m}\cap \mathcal{S}_n) \sim (\phi_0^{\mathcal{S}_n}, B_0^{\mathcal{S}_n}).$$ Notice that
\begin{align*}
 \D_{\mathcal{S}_n}^{\drk} (\phi_0^{\mathcal{S}_m} \vert_{\mathcal{S}_n}, \phi_t) &= \D_{\mathcal{S}_n}^{\drk} (\phi_0^{\mathcal{S}_m}, \phi_t)\\
 &\leq \D_{\mathcal{S}_m}^{\drk} (\phi_0^{\mathcal{S}_m}, \phi_t).
\end{align*}
because $\mathcal{S}_n \subset \mathcal{S}_m$.  Since $\D_{\mathcal{S}_m}^{\drk} (\phi_0^{\mathcal{S}_m}, \phi_t)\rightarrow 0$ as $t\to 0$ we know that $\D_{\mathcal{S}_n}^{\drk} (\phi_0^{\mathcal{S}_m} \vert_{\mathcal{S}_n}, \phi_t) \rightarrow 0$ as $t\to 0$.  From Lemma \ref{lem_rkcomplete2} we know that such a limit with domain a subset of $\mathcal{S}_n$ is unique up to $\sim$.  Thus we conclude that $\phi_0^{\mathcal{S}_m} \vert_{\mathcal{S}_n} \sim \phi_0^{\mathcal{S}_n}.$ This means that the symmetric difference of their domains has zero volume, so $\v{(B_0^{\mathcal{S}_n}\vartriangle B_0^{\mathcal{S}_m})\cap\mathcal{S}_n}=0$, and also they are equal almost everywhere on the overlap of their domains.  So now we can define $B_0 = \bigcup_{n=1}^\infty B_0^{\mathcal{S}_n}$ and $\phi_0: B_0 \rightarrow \R^k$ almost everywhere by $$\phi_0 (x) = \phi_0^{\mathcal{S}_n} (x)\text{ where } x \in \mathcal{S}_n$$ and this is well defined.  Since $\phi_t \xrightarrow{\D_{\mathcal{S}_n}^{\drk}} \phi_0$ as $t\to 0$ for all $\mathcal{S}_n$ in a compact exhaustion of $M$ we know by definition that $\phi_t \rightd \phi_0$ as $t\to 0$.
\end{proof}

Now we are ready to prove that $(\m,\D)$ is complete.

\begin{lemma}\label{lem_complete}
 Suppose that $\Snfull$ is a nested exhaustion of $M$ by finite measure sets and that $\dist$ is a metric on $N$ induced by a Riemannian metric. Then  $\big(\m, \Dfull\big)$ is complete if and only if $(N, \dist)$ is complete.
\end{lemma}

\begin{figure}
 \centering
 \includegraphics[height=160pt]{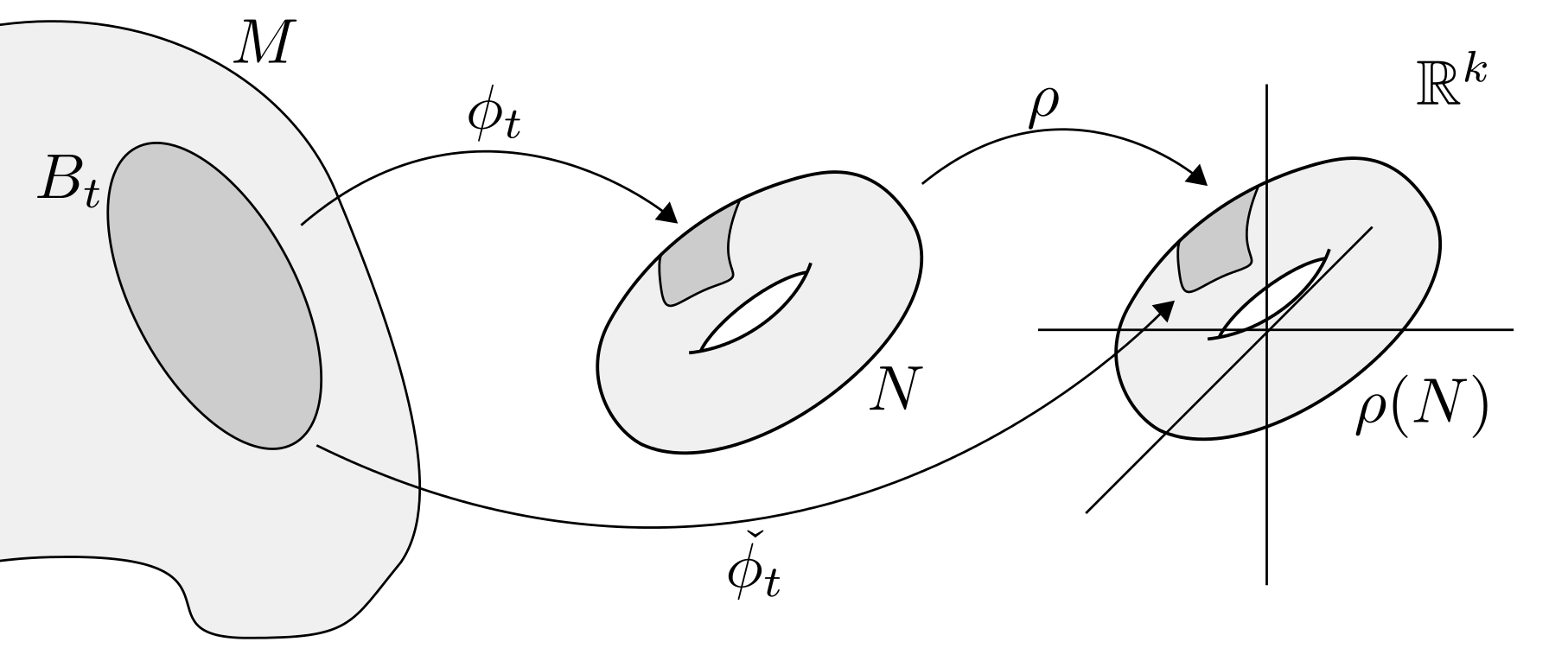}
 \caption{\{$\check{\phi_t}\}_{\tI}$ is a family of maps into $\R^k$.}
 \label{fig_isomemb}
\end{figure}

\begin{proof}

It is sufficient to show that Cauchy families indexed by $\I$ converge as $t\to 0$. First assume that $(N,d)$ is complete and let $\fI \in \F$ be Cauchy as $t\to 0$.  Now, by the Nash embedding theorem ~\cite[Theorem 3]{Nash} we know there exists an isometric embedding $\rho: N \rightarrow \R^k$ for some $k \in \N$.  In fact, since $N$ is a complete Riemannian manifold we can choose $\rho$ to have a closed image ~\cite[Theorem 0.2]{Muller}.

Let $d_N$ denote the distance on $N$ induced by the metric and let $d_{\R^k}$ denote the standard distance on $\R^k$.  Notice for $y_1, y_2 \in N$ we have that \begin{equation}\label{eqn_dineq}
d_{\R^k}(\rho(y_1), \rho(y_2)) \leq d_N (y_1, y_2).
\end{equation}
(See Remark \ref{rmk_iso}). Define $$\check{\phi_t}:=\rho \circ \phi_t:B_t \rightarrow \R^k$$ as is shown in Figure \ref{fig_isomemb}. From Equation \eqref{eqn_dineq} above we know that $$\DSrk (\check{\phi_t}, \check{\phi_s}) \leq \DS (\phi_t, \phi_s)$$ for all compact $\mathcal{S} \subset M$ so we can conclude that $\{ (\check{\phi_t}, B_t \}_\tI$ is also Cauchy with respect to $\D$.  By Lemma \ref{lem_rkcomplete3} we know that there exists some $\check{\phi_0}: B_0 \rightarrow \R^k$ such that $\check{\phi_t} \rightd \check{\phi_0}$ as $t\to 0$ and by Lemma \ref{lem_closedlimit} we can conclude, up to measure zero corrections, that $$\check{\phi_0}(B_0) \subset \rho(N).$$  Thus we may define $$\phi_0:=\rho^{-1} \circ \check{\phi_0}: B_0 \rightarrow N.$$  By Lemma \ref{lem_dchecktop} we know that $\check{\phi_t} \rightd \check{\phi_0}$ as $t\to 0$ implies that $\phi_t \rightd \phi_0$ as $t\to 0$ and so we can conclude that the Cauchy sequence converges.

It is easy to see that if $N$ is not complete then $\m$ is not complete.  Consider a sequence of constant functions $\{\phi_t:M\rightarrow N\}_\tI$ such that $$\phi_t(x) = y_t$$ where $y_t$ is a Cauchy family in $N$ which does not converge.
 
\end{proof}

The proof of Theorem \ref{thm_complete} follows from Proposition \ref{prop_ddist}, Lemma \ref{lem_equivriemmetric}, Lemma \ref{lem_complete}, and the fact that every manifold admits a complete Riemannian metric ~\cite[Theorem 1]{Nomizu}.

\section{Almost everywhere convergence and \texorpdfstring{$\D$}{the distance}}\label{sec_convg}
We already have a definition of convergence in distance, so in this section we will define and explore the properties of a way in which these maps can converge pointwise almost everywhere.  To talk about convergence of a family in $\F$ we must have both the domains and the mappings converge.  First, we will describe the convergence of the domains.

Let $a,b,c\in\R$ with $a<b$ and $c\in[a,b]$.  Now let $\{B_t \subset M\}_\tab$ be a collection of measurable subsets of $M$.  Recall the limit inferior and limit superior of a family of sets, given by $$\lic (B_t) := \bigcup_{\delta\in (0,1)} \left( \bigcap_{\substack{t \in (a,b),\\ \left|t-c\right|<\delta}} B_t \right)$$ and $$\lsc (B_t) := \bigcap_{\delta\in \I} \left( \bigcup_{\substack{t \in (a,b),\\ \left|t-c\right|<\delta}} B_t \right)$$ respectively. So the limit inferior of the family is the collection of all points which are eventually in every $B_t$ as $t\to c$ and the limit superior is the collection of all points which are not eventually outside of every $B_t$.  Clearly it can be seen that $\li(B_t) \subset \ls(B_t)$. We say that the family converges if these two sets only differ by a set of measure zero.  That is,
\begin{definition}\label{def_setconvg}
 Let $a,b,c\in\R$ with $a<b$ and $c\in[a,b]$ and let $\fab \in \F$. If $$\vplain\left\{\lsc(B_t) \setminus \lic(B_t)\right\}=0$$ we say that the collection of sets $\{B_t\}_{\tab}$ {\it converges to $\lic(B_t)$ as $t\to c$} or $\fab$ has {\it converging domains as $t\to c$}. Furthermore, if $\{[\phi_t, B_t]\}_{\tab} \in \Ft$ has converging domains for one choice of representative we say it has {\it converging domains}.
\end{definition}

\begin{remark}
 Notice that any nested family of subsets will converge by this definition.  For $a,b\in\R$ with $a<b$ let $\{B_t\}_\tab$ be a family of subsets such that for $s,t\in(a,b)$ we have that $s<t$ implies $B_t\subset B_s$. Then $$\lia B_t = \lsa B_t = \bigcup_\tab B_t.$$
\end{remark}
\begin{remark}
 Notice that if $\ftab \in \Ft$ has converging domains as $t\to c$ (for $a,b,c\in\R, a<b, c\in[a,b]$) then we can always choose some collection of representatives $\{(\phi_t', B_t') \in [\phi_t,B_t]\}_\tab$ such that $\li B_t' = \ls B_t'$ where both limits are taken as $t\to c$.
\end{remark}

Now that we understand the convergence of domains we are prepared to describe almost everywhere convergence in $\mp$. Let $a,b,c\in\R$ with $a<b$ and $c\in[a,b]$.  Notice that if $x \in \lic(B_t)$ then there exists some $\delta > 0$ such that if $t\in (a,b)$ and $\left|t-c\right|<\delta$ then $x \in B_t$.  This means that $\phi_t(x)$ exists for such $t$ so we may ask if $\{\phi_t(x)\}_{t \in (a,b)\cap (c-\delta,c+\delta)}$ converges as a family of points in $N$ as $t\to c$.  If it does converge than we have a limit $$\lim_{t \rightarrow c} \phi_t (x)$$ and thus we arrive at Definition \ref{def_ptwsae}.


\begin{remark}
 Here it is important to notice that the limit $(\phi, B)$ from Definition \ref{def_ptwsae} is not unique in $\mp$ but by Corollary \ref{cor_uniquelimitsplain} we know it does represent a unique element in $\m$.  Furthermore, given $\ftab \in \Ft$ we can create a family in $\F$ by making a choice of representative for each $\tab.$ If a choice exists such that the resulting family in $\F$ converges than we say that $\ftab$ converges almost everywhere pointwise. The limit could potentially depend on the choice of representatives, but Corollary \ref{cor_uniquelimitstilde} shows that any limit computed in this way gives the same element of $\m$. In such a case we would write $[\phi_t] \rightae [\phi_0]$ as $t\to c$. Note that the existence of one choice of representatives which converges does not guarantee that all choices will converge. 
\end{remark}

We are now ready to prove Theorem \ref{thm_convg}.

\begin{proof}[Proof of Theorem \ref{thm_convg}]
 It is sufficient to prove for families indexed by $\I$ and limits as $t\to 0$. Let $\Snfull$ be a nested exhaustion of $M$ by finite volume sets, $(\phi, B)\in\mp$, and $\fI\in\F$ such that $\phi_t \rightae \phi$ as $t\to 0$. For the duration of this proof let $\li (B_t)$ denote $\liz (B_t)$ and $\ls B_t$ denote $\lsz B_t$. 
 
Recall that for $x \in B$ we have that $x\in \li B_t$ and $\phi_t(x) \to \phi(x)$ as $t\to 0$ by Definition \ref{def_ptwsae}.  Thus $$\lim_{t \to 0} \pt (x) = \lim_{t\to 0} \text{min}\{1,d(\phi_t(x), \phi(x))\} = \text{min}\left\{1,d\left(\lim_{t\to 0} \phi_t (x), \phi(x)\right)\right\}=0.$$ Also notice that for any $x \in M \setminus \ls B_t$ we know that $x \notin B$ and also for small enough $t$ we know $x \notin B_t$.  That is, there exists some $T\in(0,1)$ such that $t<T$ implies that $x\notin B_t$ so for such $t$ we have that $x\notin B \cup B_t$. This means that for $t<T$ we have that $\pt (x) = 0.$ Thus $$\lim_{t \to 0} \pt (x) = 0$$ for any $x\in M\setminus \ls B$ as well. Every $x \in \mathcal{S}$ must either
\begin{enumerate}
 \item be in $B$ or $M \setminus \ls B_t$ and thus satisfy $\lim \pt (x) = 0$ as $t\to 0$;
 \item be in $\ls B_t \setminus B_0$, which is a set of measure zero.
\end{enumerate}
This means that $\pt \to 0$ as $t\to 0$ pointwise almost everywhere. Also notice that each $\pt$ is bounded by the constant function 1, which is integrable on $M$ because $\vnu{M} = 1.$ These two facts allow us to invoke the Lebesgue Dominated Convergence Theorem to conclude that $$\lim_{t\to 0}\Dfull (\phi_t, \phi) = \lim_{t \to 0} \varint_M \pt\, \dn = \varint_M \lim_{t \to 0} \pt \, \dn=0.$$
\end{proof}

\begin{remark}
 Notice that the converse of Theorem \ref{thm_convg} does not hold.  We know because of Example \ref{ex_wave} in which the family converges in $\D$ but not pointwise almost everywhere.
\end{remark}

The following two results are a consequence of Theorem \ref{thm_convg} and the fact that $(\m, \Dfull)$ is a metric space.

\begin{cor}\label{cor_uniquelimitstilde}
 Almost everywhere pointwise limits of families in $\Ft$ are unique in $\m$. That is, let $a,b,c\in\R$ with $a<b$ and $c\in[a,b]$.  Now suppose $\ftab \in \Ft$, $(\phi_t^1, B_t^1), (\phi_t^2, B_t^2) \in [\phi_t, B_t]$ for $\tab$, and $(\phi^1,B^1), (\phi^2,B^2)\in\mp$ such that $(\phi_t^i,B_t^i)\rightae(\phi^i, B^i)$ as $t\to c$ for $i=1,2$. Then $[\phi^1, B^1] = [\phi^2, B^2]$ in $\m$.
\end{cor}

\begin{proof}
 Let $\ftab, (\phi_t^1, B_t^1), (\phi_t^2, B_t^2), (\phi^1, B^1),\text{ and }(\phi^2, B^2)$ be as in the statement of the Corollary. Thus for any choice of a nested exhaustion of $M$ by finite volume sets $\Snfull$, a complete metric $\dist$ on $N$ which is induced by a Riemannian metric, and $\tab$ we have that $$0\leq\Dfull (\phi^1, \phi^2) \leq \Dfull(\phi^1, \phi_t^1) + \Dfull(\phi_t^1, \phi_t^2) + \Dfull(\phi_t^2, \phi^2).$$ The middle term on the right side is zero because $(\phi_t^1, B_t^1) \sim (\phi_t^2, B_t^2)$ and the remaining terms both approach zero as $t\to c$ because $(\phi_t^i,B_t^i)\rightd(\phi^i, B^i)\in\m$ as $t\to c$ by Theorem \ref{thm_convg}.
\end{proof}

\begin{cor}\label{cor_uniquelimitsplain}
 Almost everywhere pointwise limits of families in $\F$ are unique up to $\sim$.  That is, suppose that $a,b,c\in\R$ with $a<b$ and $c\in[a,b]$ and further suppose that $\fab \in \F$ and $\phi, \phi' \in \mp$. If $\phi_t \rightae \phi$ and $\phi_t \rightae \phi'$ then $\phi \sim \phi'$.
\end{cor}

\section{Families with singular limits}\label{sec_fam}

We will be considering one parameter families of mappings in $\Ft$.  For this type of family we can adapt the definition of smoothness\footnote{Despite the choice of terminology, it is unknown if this sense of smoothness implies that the family is continuous with respect to the topology on $\mp$.} from ~\cite{PeVN2012} which is visualized in Figure \ref{fig_smoothfam}.

\begin{definition}\label{def_smooth}
Let $a,b \in\R$ with $a<b$.  We say that a family of smooth maps $\fab \in \Fab$ is {\it smooth} if:
 \begin{enumerate}
  \item each element of $\{B_t\}$ is a submanifold of $M$;
  \item there exists a smooth manifold $B$ and a smooth map $g:(a,b) \times B \to M$ such that 
  \begin{enumerate}
   \item the mapping $g_t:b \mapsto g(t,b)$ is a smooth immersion;
   \item for each $t\in (a,b)$ we have $g_t(B)=B_t$.
  \end{enumerate}
  \item the map $(t,b) \mapsto \phi_t \circ g_t(b)$ is smooth.
 \end{enumerate}
\end{definition}

\begin{figure}
 \centering
 \includegraphics[height=220pt]{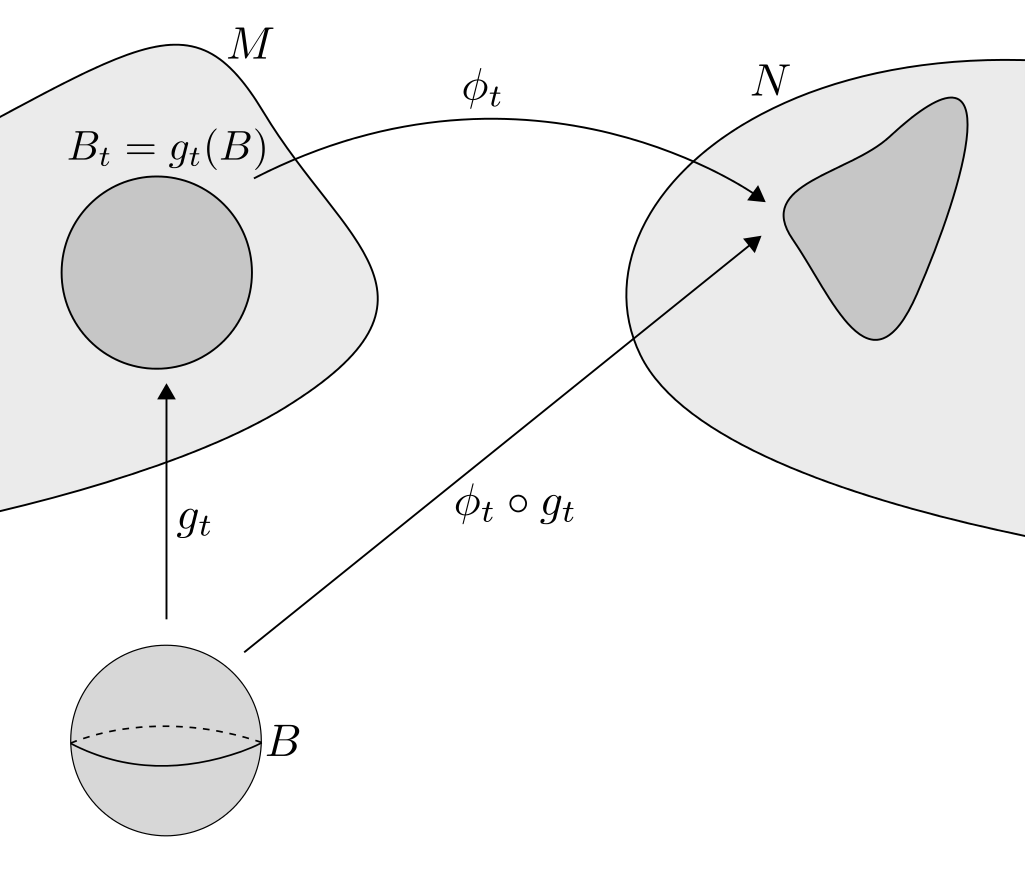}
 \caption{A figure of the relevant maps when defining a smooth family of embeddings.}
 \label{fig_smoothfam}
\end{figure}

\begin{definition}
 We say that a smooth family $\fab \in \Feab$ has a {\it singular limit} if either
 \begin{enumerate}
  \item\label{case_ess} the family does not converge in $\D$ as $t\to a$;
  \item there exists some $\phi_0 \in \mp$ such that $\phi_t \rightd \phi_0$ as $t\to a$ but $[\phi_0] \notin \et$.
 \end{enumerate}
In Case \eqref{case_ess} we say that the singularity is {\it essential} (because Theorem \ref{thm_r0convg} assures it cannot be removed).
\end{definition}

Recall the function $\rfull: \Fab \rightarrow [0,\infty]$ from Definition \ref{def_rpert}.  This function quantifies how far a family is from converging by measuring how much each embedding must be changed in order to create a new family which does converge.  It is straightforward to show that $r$ is surjective\footnote{To conclude that $r$ is actually surjective we must also show that $r=\infty$ is possible.  This is clear when a family such as $\phi_t:\I\to\R$, $\phi_t(x)=\left(\nicefrac{1}{t}\right) \text{sin}(\nicefrac{1}{t})$ is considered.}.

\begin{prop}
 For any $q \in \R$ there exists some choice of manifolds $M$ and $N$, an exhaustion $\Snfull$ of $M$, a distance $d$ induced by a complete Riemannian metric on $N$, $a,b\in\R$ such that $a<b$, and a smooth family $\fab \in \F$ for which $\rfull (\fab) = q.$
\end{prop}

\begin{proof}
 From the existence of families of embeddings which do converge we know that $0$ is in the range of $r$.  Pick some $q > 0$ and let $\phi_t: (0,3q) \rightarrow \R$ for $\tI$ via $$\phi_t(x) = \frac{x}{9q} + \frac{1}{3}\text{sin}(\nicefrac{1}{t}).$$ So in this case $B_t = (0,3q)$ for all $t$, $a=0$, $b=1$, $M = (0,3q)$ with the usual measure inherited from $\R$, and $N = \R$ with the usual distance.  Since $M$ is finite throughout this example let $\D:=\D_{\{M\}}^d$ and $r := r_{\{M\}}^d$ where $d$ is the standard distance on $\R$. Notice that if we perturbed this family to converge to some limit which did not have $(0,3q)$ as its domain we could change the domain of the limit to $(0,3q)$ and have a smaller perturbation.  So we can assume that the domain of the limit is $(0,3q)$.  Suppose that we wanted to change this family so it converged to some map $\phi_0 : (0,3q) \rightarrow \R$.  We can see that the $\phi_t$ oscillate to the left and right, so let $$\phi_L (x)= \frac{x}{9q}-\nicefrac{1}{3}$$ and $$\phi_R (x) = \frac{x}{9q}+\nicefrac{1}{3}.$$  Now let $$l_n = \frac{2}{(4n+1)\pi} \text{ and } r_n = \frac{2}{(4n+3)\pi}$$ so that $\phi_{l_n} = \phi_L$ and $\phi_{r_n} = \phi_R$ for all $n \in \N$.  Notice $$d(\phi_L (x), \phi_R(x))=\nicefrac{2}{3}$$ for all $x\in (0,3q)$ so $$d(\phi_L (x), \phi_0(x)) + d(\phi_0(x), \phi_R(x)) \geq \nicefrac{2}{3}.$$  Clearly this implies that $$\text{min}\{1,d(\phi_L (x), \phi_0(x))\} + \text{min}\{1,d(\phi_0(x), \phi_R(x))\} \geq \nicefrac{2}{3}$$ and so integrating each side over $(0,3q)$ gives $$\mathcal{D}(\phi_L, \phi_0) + \mathcal{D}(\phi_0, \phi_R) \geq 2q$$ so one of the two terms must be greater than or equal to $q$.  Without loss of generality suppose that $\mathcal{D}(\phi_L, \phi_0) \geq q$.  In such a case choose any $\varepsilon > 0$ and find some $T\in \I$ such that $t < T$ implies $\mathcal{D}(\widetilde{\phi_t}, \phi_0) < \varepsilon$ where $\{\widetilde{\phi_t}\}$ is any family which converges to $\phi_0$.  Then pick some $n\in\N$ such that $l_n < T$ and let $t=l_n$.  Now $$\mathcal{D} (\phi_t, \widetilde{\phi_t}) + \mathcal{D}(\widetilde{\phi_t}, \phi_0) \geq \mathcal{D}(\phi_t, \phi_0)$$ so $\mathcal{D}(\phi_t, \widetilde{\phi_t}) \geq q-\varepsilon$ for all $\varepsilon>0$.  This allows us to conclude that $r(\f) \geq q.$
 
 Now let $\widetilde{\phi_t}:(0,3q)\rightarrow \R$ with $\widetilde{\phi_t} (x) = \frac{x}{9q}$ be a family of maps which is clearly smooth and has limit $\phi_0 (x) = \frac{x}{9q}.$  Now notice $$\mathcal{D}(\phi_t, \widetilde{\phi_t}) = \varint_{\mc{(0,3q)}} \text{min}\{1, d(\phi_t, \widetilde{\phi_t})\}\, \dv = q \left| \text{sin}(\nicefrac{1}{t})\right| \leq q$$and it is important to notice that $\mathcal{D}(\phi_t, \widetilde{\phi_t})=q$ is achieved infinitely often.  Thus we know that $r(\f) \leq q$ so in fact we know that $r(\f)=q$.
\end{proof}

In the case that $\rfull (\f) = 0$ we say that the family has a {\it removable singularity with respect to $\Dfull$}\footnote{It may be true that having zero radius of convergence is independent of the chose of parameters $d$ and $\Snfull$, see Section \ref{subsec_implic}.}. Now we are prepared to prove Theorem \ref{thm_r0convg}.

\begin{proof}[Proof of Theorem \ref{thm_r0convg}]
 Suppose that $\fab \in \F$ and $r(\fab)=0$.  Fix some compact $\mathcal{S} \subset M$ and we will show that $\fab$ is Cauchy with respect to $\mathcal{D}_\mathcal{S}.$
 
 Fix some $\varepsilon > 0.$ Let $\delta = \nicefrac{\varepsilon}{4}$ and since $r(\fab)=0$ define some other family $\{(\widetilde{\phi_t^\delta}, \widetilde{B_t^\delta})\}$ such that
 \begin{enumerate}
  \item $\mathcal{D}(\phi_t, \widetilde{\phi_t^\delta}) < \delta$ for all $\tab$;
  \item $ B_t = \widetilde{B_t^\delta}$ for all $\tab$;
  \item\label{blah} there exists some $\widetilde{\phi^\delta}\in \e$ such that $\widetilde{\phi_t^\delta} \rightae \widetilde{\phi^\delta}$ as $t\to a$.
 \end{enumerate}

 From Theorem \ref{thm_convg} and item \eqref{blah} above we know that $$ \widetilde{\phi_t^\delta} \rightds \widetilde{\phi_0^\delta} $$ as $t\to a$ so we can choose some $T \in (a,b)$ such that $t < T$ implies $\mathcal{D}_\mathcal{S} (\widetilde{\phi_t^\delta}, \widetilde{\phi^\delta})<\delta.$  Finally, we can conclude that for any $t,s < T$ we have that
 \begin{align*}
  \mathcal{D}_\mathcal{S}(\phi_t, \phi_s) &\leq \mathcal{D}_\mathcal{S} (\phi_t, \widetilde{\phi_t^\delta}) + \mathcal{D}_\mathcal{S} (\widetilde{\phi_t^\delta}, \widetilde{\phi^\delta}) + \mathcal{D}_\mathcal{S} (\widetilde{\phi^\delta}, \widetilde{\phi_s^\delta}) + \mathcal{D}_\mathcal{S} (\widetilde{\phi_s^\delta}, \phi_s)\\
  &< 4 \delta = \varepsilon.
 \end{align*}
This means that $\fab$ is Cauchy as $t\to a$ for each $\mathcal{D}_\mathcal{S}$ so by Proposition \ref{prop_dconv} we know that it is Cauchy with respect to $\D$ as $t\to a$.  Finally, since $(\m,\D)$ is complete by Theorem \ref{thm_complete} we can come to the first conclusion of this Theorem.

Now we will show the second claim.  Suppose that the domains satisfy the required property for $T \in (a,b)$ and that $\phi_t \rightd \phi_0$ as $t\to a$. Fix $\varepsilon >0$ and find some $T_1 \in (a, T)$ such that $s, t < T_1$ implies that $\D(\phi_t, \phi_s) < \varepsilon.$ Now let $\mathcal{b}: (a,b) \to [0,1]$ be a smooth bump function such that $\mathcal{b}(t)=0$ for $t\geq T_1$ and $b(t) = 1$ for $t < \nicefrac{T_1+a}{2}.$  Now define $f:(a,b) \to [\nicefrac{T_1+a}{2}, b)$ via $$f(t) = \big( 1-\mathcal{b}(t) \big) t + \mathcal{b}(t) \frac{T_1+a}{2}.$$ Finally let $$\widetilde{\phi_t} = \phi_{f(t)}\vert_{B_t}$$ and notice that this is a smooth family satisfying $\widetilde{\phi_t} \rightae \phi_{\nicefrac{T_1}{2}}$ as $t\to a$. By the choice of $T_1$ we can see that for all $\tab$ we have $\D(\phi_t, \widetilde{\phi_t})<\varepsilon.$ Also, because of the requirement on the domains we know that $B_t\subset B_{f(t)}$ and thus $\widetilde{\phi_t}: B_t \to N$ is defined on all of $B_t$.
 
\end{proof}

\begin{remark}
It is natural to wonder if $r(\fab)=0$ implies the family must in fact converge pointwise almost everywhere in $\m$.  The answer to this question is no; again consider Example \ref{ex_wave}. The functions in Example \ref{ex_wave} converge in $\D$ and all have the same domain so we know that $\rfull=0$ for these functions, but we also know that they do not converge pointwise almost everywhere. 
\end{remark}

\section{Final remarks}\label{sec_ques}

\subsection{Approaches to prove a converse to Theorem \ref{thm_r0convg}}Now we have set up all of the machinery to begin to explore the converse of Theorem \ref{thm_r0convg} in the case that the domains are not restricted to shrink or stabilize eventually. That is, we will outline some potential avenues to answer the following question.

\begin{question}\label{ques_1}
 Is it true that $\fab \rightd \phi_0$ implies that $\rfull(\fab) = 0$?\\
\end{question}

There are two approaches in the general case: we can attempt to extend embeddings or we can smooth singular limits by understanding the singularities locally.

\subsubsection{Extending embeddings to remove singularities}
This method extends the idea used to prove the partial converse direction of Theorem \ref{thm_r0convg} given in the statement of the theorem.  The idea is that if $\fab \rightd \phi_0$ as $t\to c$ (for $a,b,c\in\R, a<b, c\in[a,b]$) then in order to get an $\varepsilon$-perturbation of $\fab$ we choose some $T \in (a,b)$ such that $s,t < T$ implies that $\Dfull(\phi_t, \phi_s)<\varepsilon$.  Then, just as in the proof of Theorem \ref{thm_r0convg}, we must smoothly change the family so that $t<\frac{T+a}{2}$ implies that $\widetilde{\phi_t} = \phi_{\frac{T+a}{2}}$.  The difficultly here is dealing with the domains.  If $\li{B_t} \not\subset B_{\frac{T}{2}}$ then this idea we have outlined will not define an embedding with domain all of $B_0$, so this embedding would have to be extended.  It is important to notice that $\v{B_0 \vartriangle B_{\frac{T+a}{2}}}<\varepsilon$ and so the embedding can be defined in any way on the extension, as long as it does not change on $B_{\frac{T+a}{2}}$.  Thus, this questions comes down to asking when an embedding of some subset of $M$ can be extended to a larger domain in $M$. Extending embeddings or smooth maps has been of independent interest for many years.  See for example the Tietze Extension Theorem ~\cite[Theorem 4.16]{Folland}, the Whitney Extension Theorem ~\cite[Theorem I]{Whitney1934}, the Extension Lemma ~\cite[Lemma 2.27]{Lee2006}, and for a collection of more recent work in extension problems see ~\cite{Brudnyi2012}.

\subsubsection{Removing singularities locally}
The basic strategy is the following.  Suppose for $a,b,c\in\R, a<b, c\in[a,b]$ that $\fab \in \F$ satisfies $\phi_t \rightd \phi$ as $t\to a$ for some $\phi \in \mp$ and suppose further that $\mathcal{S} \subset B$ is a closed subset of $M$ containing all of the singular points of the limiting map $\phi$ and that eventually $\mathcal{S} \subset B_t$ for all $t$.  That is, we assume that $$\phi \vert_{B \setminus \mathcal{S}} : B \setminus \mathcal{S} \se N$$ is an embedding and there exists some $T\in(a,b)$ such that $t<T$ implies $\mathcal{S} \subset B_t$.  Then for some neighborhood of $\mathcal{S}$ we can define $\widetilde{\phi}$ by $\phi_{t_0}$ restricted to that neighborhood for some small enough $t_0 \in (a,b)$.  Then to define $\widetilde{\phi}$ outside of a slightly larger neighborhood of $\mathcal{S}$ we simply use $\phi$ unchanged.  Then we must connect these two pieces in a way which makes the result an embedding.  Finally each $\phi_t$ can then be changed on a neighborhood of $\mathcal{S}$ to converge to $\phi_{t_0}$ and outside of that neighborhood they converge to $\phi=\widetilde{\phi}$ already. This idea is shown in Figure \ref{fig_strategy}. The difficulty comes when we must connect the two embeddings; it is well known that partition of unity type arguments can be used to smoothly transition between two smooth maps ~\cite{Lee2006} but in this case we must also preserve the embedding structure.

\begin{figure}
 \centering
 \includegraphics[height=160pt]{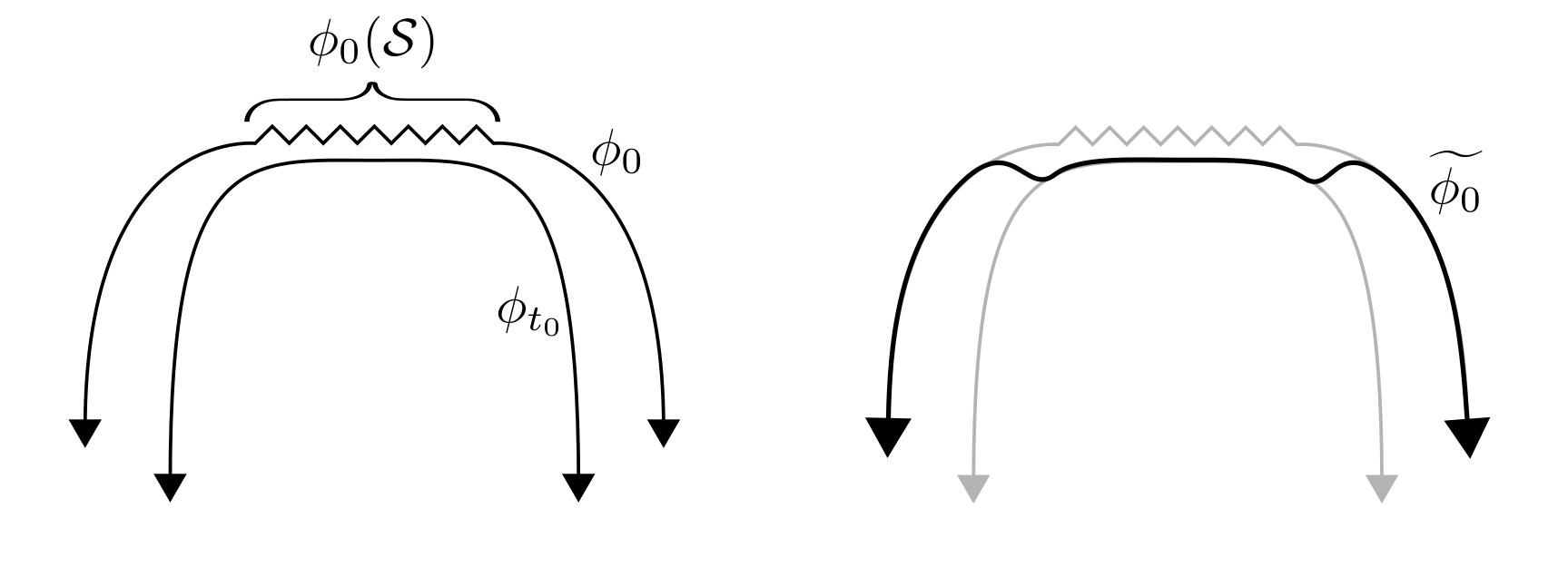}
 \caption{The strategy is to connect the embedding $\phi_{t_0}$ with the map $\phi$ which is an embedding away from $\mathcal{S}.$ In this way we are able to avoid the singular part of $\phi$ while only changing it slightly on a small set.}
 \label{fig_strategy}
\end{figure}

\subsection{Implications of a positive answer to Question \ref{ques_1}}\label{subsec_implic}

If the answer to Question \ref{ques_1} were yes, then there are several implications.  First, we will have a new characterization of families with removable singularities, namely these are exactly the families which converge in $\D$.  Second, and most importantly, there is then an easy proof that $\rfull(\f)=0$ does not depend on the choices of $\Snfull$ and $\dist$.  The proof is the following:

Let $\Snfull$, $\{\mathcal{S}_n'\}_{n=1}^\infty$, $\dist$, and $\dist'$ be choices of finite exhaustion and metric.  Suppose that $a,b\in\R$ with $a<b$ and $\fab\in\F$ is a smooth family such that $\rfull(\fab)=0$.  Then by Theorem \ref{thm_r0convg} we know that $\lim \Dfull (\phi_t, \phi) = 0$ as $t\to a$ for some $\phi \in \mp$.  By Theorem \ref{thm_complete} this means that $\lim \mathcal{D}_{\{\mathcal{S}_n'\}}^{d'} (\phi_t, \phi) = 0$ as $t\to a$ and thus by the assumed positive answer to Question \ref{ques_1} we know that $r_{\{\mathcal{S}_n'\}}^{d'} (\fab) = 0.$

\subsection{Further questions}
It would be interesting to study Question \ref{ques_1} restricted to a specific type of embedding.  For example, thinking back to the original motivation from Section \ref{sec_intro}, one could consider whether this is true for the collection of symplectic embeddings\footnote{Resolving singular points of symplectic manifolds is related to this in spirit and has been studied extensively such as in ~\cite{McWo1995}} where the original smooth family $\fab$ consists exclusively of symplectic embeddings and the perturbed family $\ftildeab$ from the definition of the radius of convergence is also required to be symplectic. Symplectic manifolds have been shown to admit a high degree of flexibility (see for example Moser's Theorem ~\cite{Mo1965} or Darboux's Theorem ~\cite{daSi2008}) although Gromov's nonsqueezing theorem ~\cite{Gr1985} represents a level of rigidity that symplectic embeddings do need to respect. One could also  consider the case of isometric embeddings of Riemannian manifolds, even in the case of $\mp(\R, \R^2)$.  Clearly studying further types of embeddings would be enlightening as it would allow us to gain a greater understanding of the rigidity of these structures.  Indeed, it is the purpose of this paper to create a foundation off of which many types of families of embeddings may be studied.\\

{\it Acknowledgements.} The author was supported by the National Science Foundation under agreement No. DMS-1055897.  He is very grateful to his advisor \'Alvaro Pelayo both for originally suggesting the study of this topic and also for many helpful discussions. 


\bibliographystyle{amsplain}
\bibliography{smooth_sing}

{\small
  \noindent
  \\
  {\bf Joseph Palmer} \\
  Washington University,  Mathematics Department \\
  One Brookings Drive, Campus Box 1146 \\
  St Louis, MO 63130-4899, USA.\\
  {\em E\--mail}: \texttt{jpalmer@math.wustl.edu} \\
}   

\end{document}